\documentclass[12pt]{article}
\usepackage{amssymb,amsmath, amsthm}
\usepackage{graphicx}

\newtheorem{thm}{Theorem}

\newtheorem*{propo}{Proposition}
\newtheorem{cor}{Corollary}

\newtheorem{lemma}{Lemma}

\newenvironment{defin}{\medskip\noindent{\sc
Definition}. }{\goodbreak\medskip}
\newenvironment{nota}{\medskip\noindent{\sc
Notations}. }{\goodbreak\medskip}
\newenvironment{remk}{\noindent{\sc
Remarks}. }{\goodbreak\vskip10pt}
\newtheorem{prop}{Proposition}

\def\demo{\medskip\goodbreak\noindent
     \hbox{\sc Proof \kern .3em}\ignorespaces}%
  \def \qedbox{$\square$}%
  \def \qed{\hglue1mm\hfill{\ifmmode\qedbox
     \else\unskip\ \hglue0mm\hfill\qedbox\medskip
      \goodbreak\fi}}%
\def\enddemo{\qed\goodbreak\vskip10pt}%

\def\qed{\hglue1mm\hfill\raise -2pt\hbox{\vrule\vbox to 10pt{\hrule width
4pt
                  \vfill\hrule}\vrule}}

\newcommand{\T}{\mathbb {T}}
\newcommand{\U}{\mathbb {U}}

\newcommand{\A}{\mathbb {A}}
\newcommand{\esse}{\mathbb {S}}
\newcommand{\R}{\mathbb {R}}

\newcommand{\Z}{\mathbb {Z}}
\newcommand{\N}{\mathbb {N}}
\newcommand{\Nc}{\mathcal {N}}

\newcommand{\Vc}{\mathcal {V}}

\newcommand{\Pc}{\mathcal {P}}

\newcommand{\Hc}{\mathcal {H}}

\newcommand{\Ic}{\mathcal {I}}

\newcommand{\Mc}{\mathcal {M}}

\newcommand{\Lc}{\mathcal {L}}
\newcommand{\Fc}{\mathcal {F}}
\newcommand{\Ac}{\mathcal {A}}

\renewcommand{\Im}{\mathrm{Im}}

\begin{document}
\title{Lyapunov exponents of minimizing measures  for globally positive diffeomorphisms in all dimensions}
\author{M.-C. ARNAUD %\date{}
\thanks{ANR-12-BLAN-WKBHJ}
\thanks{Avignon Universit\'e , Laboratoire de Math\'ematiques d'Avignon (EA 2151),  F-84 018 Avignon,
France. e-mail: Marie-Claude.Arnaud@univ-avignon.fr} 
\thanks{membre de l'Institut universitaire de France}}
\maketitle
\abstract{\noindent \sl  The globally positive diffeomorphisms of the $2n$-dimensional annulus are important because they represent what happens close to a completely elliptic periodic point of a symplectic diffeomorphism where the torsion is positive definite.\\
For these globally positive diffeomorphisms, an Aubry-Mather theory was developed by Garibaldi \& Thieullen that provides the existence of some minimizing measures.\\
Using the two Green bundles $G_-$ and $G_+$  that can be defined along the support of these minimizing measures, we will prove that there is a deep link between:
\begin{enumerate}
\item[$\bullet$] the angle between $G_-$ and $G_+$ along the support of the considered measure $\mu$;
\item[$\bullet$] the size of the smallest positive Lyapunov exponent of $\mu$;
\item[$\bullet$] the tangent cone to the support of $\mu$.
\end{enumerate}

} 

   \section*{Introduction} 
At the end of the 19th century, motivated by the restricted 3-body problem, H.~Poincare introduced the study of the area preserving  diffeomorphisms near an elliptic fixed point. \\
Then, in the '30s, Birkhoff began the study  of the exact symplectic twist maps~: after a symplectic change of coordinates (action-angle), these maps represent what happens near an elliptic fixed point of a generic area preserving diffeomorphism (see \cite{Bir1}). \\
In the '80s, S.~Aubry \& P.~Le~Daeron and J.~Mather proved the existence of invariant minimizing measures for these twist maps (see \cite{ALD} and \cite{Mat1}). 
As proved by P.~Le~Calvez, these minimizing measures are in general hyperbolic (see \cite{LC1}).  For such minimizing measures, I proved in \cite{Arn1} that there is a link between the fact that they are hyperbolic and the regularity in some sense of their support and I proved in \cite{Arn4} that there is a link between the size of the Lyapunov exponents and the mean angle of the Oseledet's splitting when the minimizing measure is   hyperbolic. A fundamental tool to obtain such results is the pair of {\em Green bundles}, that are two bundles in lines that are defined along the support of the minimizing measures.

A natural question is then: what happens in higher dimension? \\
Let us explain what is a twist map in this setting (see for example  \cite{Gol1} or    \cite{Arn3}).  
\begin{nota}
The $2n$-dimensional annulus is $\A_n=\T^n\times \R^n$   endowed  with its usual symplectic form $\omega$. More precisely, if $q=(q_1, \dots, q_n)   \in\T^n$ and $p=(p_1, \dots , p_n)\in\R^n$ then $\displaystyle{\omega=dq\wedge dp=\sum_{i=1}^ndq_i\wedge dp_i}$.  .\\
       Let us recall that a diffeomorphism $f$ of $\A_n$ is {\em symplectic} if it preserves the symplectic form: $f^*\omega=\omega$. \\
   We denote by $\pi:\A_n\rightarrow \T^n$ the    projection $(q, p)\mapsto q$.\\
   At every $x=(q,p)\in \A_n$, we define the vertical subspace $V(x)=\ker D\pi(x)\subset T_x \A_n$ as being the tangent subspace at $x$ to the fiber $\{ q\}\times \R^n$.
\end{nota}

\begin{defin}
 A {\em globally positive} diffeomorphism of $\A_n$ is a symplectic $C^1$-diffeomorphism $f:\A_n\rightarrow \A_n$ that is homotopic to ${\rm Id}_{\A_n}$ and that has a lift  $F: \R^n\times \R^n\rightarrow \R^n\times \R^n$  that admits a $C^2$ generating function $S: \R^n\times \R^n\rightarrow \R$ such that:
 \begin{enumerate}
 \item[$\bullet$] there exists $\alpha>0$ such that: $\frac{\partial^2 S}{\partial q\partial Q}(q, Q)(v, v)\leq -\alpha\| v\|^2$;
 \item[$\bullet$] $F$ is implicitly given by:
 $$F(q,p)=(Q, P)\Longleftrightarrow \left\{\begin{matrix}p=-\frac{\partial S}{\partial q}(q, Q)\\ P=\frac{\partial S}{\partial Q}(q, Q)
 \end{matrix}\right.$$
 where $ \|.\|$ is the usual Euclidean norm in $\R^n$.
 \end{enumerate}
  \end{defin}
  When we use a symplectic change of basis  near a completely elliptic  periodic point  of a generic symplectic  diffeomorphism in any dimension, we obtain a Birkhoff normal form defined on a subset on  $\A_n$ by $(q, p)\mapsto (q+b.p+o(\| p\|), p+o(\|p\|)$ where the torsion $b$ is a symmetric non-degenerate matrix.  When $b$ is positive definite, this normal form is a a globally positive diffeomorphims on some bounded subannulus $\T^n\times [a, b]^n$ (see for example \cite{Mos2} or \cite{Arn6}).
  
    \begin{remk}
 If $f$,   $F$ satisfy the above hypotheses, the restriction to any fiber $\{q\}\times \R^n$ of $\pi\circ F$ and $\pi\circ F^{-1}$ are diffeomorphisms. 
 Moreover, for every $k\geq 2$, $q_0, q_k\in\R^n$, the function $\displaystyle{\hat\Fc:(\R^n)^{k-1}\rightarrow \R}$ defined by $\displaystyle{\hat\Fc(q_1, \dots, q_{k-1})=\Fc(q_0, \dots, q_k)=\sum_{j=1}^kS(q_{j-1}, q_j)}$ has a minimum, and at every critical point for $\hat\Fc$, the following sequence is a piece of orbit for $F$:
 $$(q_0, -\frac{\partial S}{\partial q}(q_0, q_1)), (q_1,  \frac{\partial S}{\partial Q}(q_0, q_1)), (q_2,    \frac{\partial S}{\partial Q}(q_1, q_2)), \dots , (q_k,  \frac{\partial S}{\partial Q}(q_{k-1}, q_k)).$$
 \end{remk}
 In the 2-dimensional case ($n=1$), J.~Mather and Aubry \& Le Daeron proved in \cite{ALD} and  \cite{Mat1} the existence of orbits $(q_i,p_i)_{i\in\Z}$ for $F$ that are {\em globally minimizing}. This means that for every $\ell\in\Z$ and every $k\geq 2$, $(q_{\ell+1}, \dots, q_{\ell+k-1})$ is minimizing the function $\hat\Fc$ defined by:
 $$\hat\Fc(q_{\ell+1}, \dots, q_{\ell+k-1})=\sum_{i=\ell+1}^{k} S(q_{i-1}, q_i).$$
 Then each of these orbits $(q_i, p_i)_{i\in\Z}$ is supported in the graph of a Lipschitz map defined on a closed subset of $\T$, and there exists a  bi-Lipschitz orientation preserving homeomorphisms $h:\T\rightarrow\T$ such that $(q_i)_{i\in\Z}=(h^i(q_0))_{i\in\Z}$. Hence each of these orbits has a {\em rotation number}.
 Moreover, for each rotation number $\rho$, there exists a minimizing orbit that has this rotation number and there even exist a {\em minimizing measure}, i.e. an invariant measure  the support of whose is filled by globally minimizing orbits, such that all the orbits contained in the support have the same rotation number $\rho$. These supports are sometimes called {\em Aubry-Mather sets}.\\
 
 For the globally positive diffeomorphisms in higher dimension,  a discrete weak KAM and an Aubry-Mather theories  were developped by E.~Garibaldi \& P.~Thieullen in \cite{GarThi}. They prove that  there exist some globally minimizing orbits and measures (the support of whose is compact and a Lipschitz graph) in $\A_n$ for all $n\geq 1$. \\
 Two Lagrangian subbundles of $T\A_n$ can be defined along the support of the minimizing measures of any   globally positive  diffeomorphism. They are called {\em Green bundles}, denoted by $G_-$ and $G_+$\footnote{Their definition is recalled in section \ref{sGreen}} and their existence is proved in \cite{BiaMac} and \cite{Arn4}.  We will prove that for any ergodic minimizing measure, the almost eveywhere dimension of the intersection of the two Green bundles gives the number of zero Lyapunov exponents of this measure:
 \begin{thm}\label{nbexp}
 Let $\mu$ be an ergodic  minimizing measure of a globally positive  diffeomorphism of $\A_n$. Let $p$ be the almost everywhere dimension of the intersection $G_-\cap G_+$ of the two Green bundles. Then $\mu$ has exactly $2p$ zero Lyapunov exponents, $n-p$ positive Lyapunov exponents and $n-p$ negative Lyapunov exponents.
 \end{thm}
 
 Then we will explain that there is a link between the angle between the two Green bundles and the size of the positive Lyapunov exponents.
  To do that, let us introduce some notations.

\begin{nota} We associate  an almost complex structure $J$ and then a Riemannian metric $(.,.)_x$ defined by: $(v,u)_x=\omega (x)(v, Ju)$ to  the symplectic form $\omega$ of $\A_n$; from now on,  we work with this fixed Riemannian metric of $\A_n$.\\ 
We choose on $ G_+(x)$ an orthonormal basis   and complete it in a symplectic basis whose last vectors are in $V(x)$.  \\
In these coordinates, $G_+$ is the graph of the zero-matrix and $G_-$ is the graphs of a negative semi-definite symmetric matrix that is denoted by $-\Delta S$.\\
In these coordinates, along the support of a minimizing measure, the image $Df.V$ of the vertical (resp. $Df^{-1}V$) is transverse to the vertical and then the graph of a symmetric matrix $S_1$ (resp. $S_{-1}$).\\
For a positive semi-definite symmetric matrix $S$ that is not the zero matrix, we decide to denote by $q_+(S)$ its smallest positive eigenvalue.
\end{nota} 
 \begin{thm}\label{sizeexp}
 Let $\mu$ be an ergodic  minimizing measure of a globally positive  diffeomorphism of $\A_n$ that has at least one non-zero Lyapunov exponent. We denote the smallest  positive  Lyapunov exponent of $\mu$ by $\lambda (\mu)$ and  an upper bound for $\| S_1-S_{-1}\|$ above ${\rm supp} \mu$ by   $C$. Then we have: 
$$\lambda (\mu)\geq\frac{1}{2}\int \log\left( 1+\frac{1}{C}q_+(\Delta S (x))\right) d\mu (x).$$
  \end{thm}
In fact, Garibaldi and Thieullen prove the existence of measures that have a stronger property than being minimizing: they are strongly minimizing\footnote{see subsection \ref{ssweakkam} for the definition}. They prove that the supports of these strongly minimizing measures are Lipschitz graphs $M$ above a compact subset of $\T^n$. In general, these graphs are not contained in a smooth graph. But we can define at every point $m\in M$  its {\em limit contingent cone} $\widetilde C_mM$ that is an extension of the notion of tangent space to a manifold \footnote{see section \ref{sscone} for the exact definition}.\\
Let us recall that we defined  in \cite{Arn5} an order $\leq$ between the Lagrangian subspaces of $T_x\A_n$ that are transverse to the vertical. If $\Lc_-$, $\Lc_+$ are two such subspaces such that $\Lc_-\leq \Lc_+$, we say that a vector $v\in T_x\A_n$ is between $\Lc_-$ and $\Lc_+$ if there exists a third Lagrangian subspace $\Lc$ such that $v\in\Lc$ and $\Lc_-\leq \Lc\leq \Lc_+$. \\
We will prove that the limit contingent cone to the support  of every strongly minimizing measure is between some {\em modified Green bundles} $\widetilde G_-$ and $\widetilde G_+$\footnote{see section \ref{sscone} for the precise definition}.
\begin{thm}\label{Tcone}
Let $\mu$ be a strongly minimizing measure of a globally positive  diffeomorphism of $\A_n$ et let ${\rm supp}\mu$ be its support. Then
$$\forall x\in {\rm supp} \mu, \widetilde G_-(x)\leq \widetilde C_x({\rm supp}\mu)\leq \widetilde G_+(x).$$

\end{thm}
Hence, the more irregular ${\rm supp}\mu$ is, i.e. the bigger the limit contingent cone is, the more distant  $\widetilde G_-$ and $\widetilde G_+$ (and thus $G_-$ and $G_+$ too)  are from each other and the larger the positive Lyapunov exponents are.

We define too a notion of $C^1$-isotropic graph (see subsection \ref{sscone}) that generalized the notion of $C^1$ isotropic manifold (for the symplectic form). Then we deduce from theorem \ref{Tcone}:
\begin{cor}\label{Cisotropic}
Let $\mu$ be an ergodic strongly minimizing measure of a globally positive diffeomorphism of $\A_n$ all exponents of whose are zero. Then ${\rm supp}\mu$ is $C^1$-isotropic almost everywhere.
\end{cor}
\subsection*{Some related results}

Theorem \ref{nbexp} is an extension of a result that we proved for the autonomous Tonelli Hamiltonians in \cite{Arn4}. The ideas of the proof are more or less the same ones as in \cite{Arn4}, but some adaptions are needed because we cannot use any continuous dependence in time.\\
A particular case of theorem \ref{sizeexp}   was proved in \cite{Arn4}:    the weak hyperbolic case, where the two Green bundles are almost everywhere transverse. Here we fill the gap by using the reduced Green bundles. \\
The inequality given in theorem \ref{Tcone} is completely new, even if an analogue to corollary \ref{Cisotropic} was given in \cite{Arn2} for Tonelli Hamiltonians.\\

 \begin{remk} 1) A discrete weak KAM theory is given in \cite{Gom1} too by D.~Gomes, but the condition used by the author there is the convexity of a Lagrangian function that is not the generating function, and this condition is different from the one we use.  But   Garibaldi \& Thieullen results can be used. 
 
 2) There exists too an Aubry-Mather theory for time-one maps of time-dependent Tonelli Hamiltonians (see for example \cite{Ber1}). Even when the manifold $M$ is $\T^n$, the time-one map is not necessarily a  globally positive diffeomorphism of $\A_n$. Moreover, except for the 2-dimensional annulus (see \cite{Mos1}),  it is unknown if a globally positive diffeomorphism is always the time-one map of a time-dependent Tonelli Hamiltonian (see theorem 41.1 in \cite{Gol1} for some partial results). In this article, we won't speak about these time-one maps and will focus on the globally positive diffeomorphisms.
 
 \end{remk}
 
\subsection*{Structure of the article}
In section \ref{sGreen}, after  explaining the construction of the classical Green bundles and the restricted Green bundles, we will prove  theorem \ref{nbexp}.\\
  We will then explain in section \ref{sAngles} that the mean angle between the two Green bundles gives the size of the smallest positive Lyapunov exponent. \\
  Section \ref{sKAM}  is devoted to some reminders in discrete weak KAM theory and to the proofs of theorem \ref{Tcone} and corollary \ref{Cisotropic}.\\
  There are two parts in the appendix. The first one is used in subsection \ref{fred} and the second one is used in subsection \ref{sscone}.
 \tableofcontents
\newpage

 \section{Green bundles}\label{sGreen}
 \subsection{Classical Green bundles}
 We recall some classical results that are in \cite{Arn4}. For the definition of the order between Lagrangian subspaces that are transverse to the vertical, see subsection \ref{sscomp} of the appendix.   Let $f:\A_n\rightarrow \A_n$ be a globally positive diffeomorphism.
  
 \begin{nota}
 If $k\in\Z$ and $x\in\A_n$, we denote by $G_k(x)$ the Lagrangian subspace $G_k(x)=(Df^k).V(f^{-k}x)$.
 \end{nota}
 
\begin{defin}
Let $x\in\A_n$ be a point the orbit of whose is minimizing. Then the sequence $(G_k(x))_{k\geq 1}$ is a strictly decreasing sequence of Lagrangian subspaces of $T_x(\A_n)$ that are transverse to $V(x)$ and $(G_{-k}(x))_{k\geq 1}$ is an increasing sequence of Lagrangian subspaces of $T_x(\A_n)$ that are transverse to $V(x)$. The two Green bundles are $x$ are the Lagrangian subspaces
$$G_-(x)=\lim_{k\rightarrow +\infty} G_{-k}(x)\quad{\rm and}\quad G_+(x)=\lim_{k\rightarrow +\infty}G_k(x).$$
\end{defin} 
It is proved in \cite{Arn4} that the two Green bundles are transverse to the vertical and verify:
$$\forall k\geq 1, G_{-k}<G_{-(k+1)}<G_-\leq G_+<G_{k+1}<G_k.$$
In general these two bundles are not continuous, but they depend in a measurable way to $x$.   Moreover, they are semicontinuous is some sense. Let us recall some properties that are proved in \cite{Arn4}. 
\begin{prop} Assume that the orbit of $x$ is minimizing. Then 
\begin{enumerate}
   \item[$\bullet$] $G_-$ and $G_+$ are invariant by the linearized dynamics, i.e. $Df.G_\pm=G_\pm\circ f$;
      \item[$\bullet$] for every compact $K$ such that the orbit of every point of $K$ is  minimizing, the two Green bundles restricted to $K$ are uniformly far from the vertical;
       \item[$\bullet$] (dynamical criterion) if the orbit of $x$ is minimizing and relatively compact in $\A_n$, if $\displaystyle{\liminf_{k\rightarrow+\infty} \| D(\pi\circ f^k)(x)v\|\leq +\infty}$ then $v\in G_-(x)$,\\
    if $\displaystyle{\liminf_{k\rightarrow+\infty} \| D(\pi\circ f^{-k})(x)v\|\leq +\infty}$ then $v\in G_+(x)$.

\end{enumerate}
\end{prop}
 An easy consequence of the dynamical criterion and the fact that the Green bundles are Lagrangian is that when there is a splitting of $T_x(T^*M)$ into the sum of a stable, a center and a unstable bundle $T_x(T^*M)=E^s(x)\oplus E^c(x)\oplus E^u(x)$, for example an Oseledets splitting, then we have
$$E^s\subset G_-\subset E^s\oplus E^c\quad{\rm and} \quad E^u\subset G_+\subset E^u\oplus E^c.$$
Let us give the argument of the proof. Because of the dynamical criterion, we have $E^s\subset G_-$. Because the dynamical system is symplectic, the symplectic orthogonal subspace to $E^s$ is   $(E^s)^\bot=E^s\oplus E^c$ (see e.g. \cite{BocVia1}). Because $G_-$ is Lagrangian, we have $G_-^\bot=G_-$. We obtain then $G_-^\bot=G_-\subset E^{s\bot}=E^s\oplus E^c$.\\
Let us note the following straightforward consequence: for a minimizing measure,    the whole information concerning the positive (resp. negative) Lyapunov exponents is contained in the restricted linearized dynamics $Df_{|G_+}$ (resp.  $Df_{|G_-}$).
 
From $E^s\subset G_-\subset E^s\oplus E^c$ and $E^u\subset G_+\subset E^u\oplus E^c$, we deduce that $G_-\cap G_+\subset E^c$. Hence $G_-\cap G_+$ is an isotropic subspace (for $\omega$) of the symplectic space $E^c$. We deduce that $\dim (E^c)\geq 2\dim (G_-\cap G_+)$. When $E^s\oplus E^c\oplus E^u$ designates the Oseledet splitting of some minimizing measure, what is proved in \cite{Arn2} is that this inequality is an equality for the Tonelli Hamiltonian flows and we will prove here the same result for the globally positive diffeomorphisms of $\A_n$.

\subsection{Reduced Green bundles}\label{fred}
The reduced Green bundles were introduced in \cite{Arn2} for the Tonelli Hamiltonian flows. We will give a similar construction.

We assume that  $\mu$ is a minimizing ergodic measure   and  that $p\in [0, n]$ is so that at $\mu$-almost every point $x$, the intersection of  the Green bundles $G_+(x)$ and $G_-(x)$ is $p$-dimensional. We deduce from the above comments that for $\mu$ almost every $x\in \A_n$:  $G_+(x)\cap G_-(x)\subset E^c(x)$ and $E^s(x)\oplus E^u(x)= \left( E^c(x)\right)^\bot\subset G_+(x)^\bot+ G_-(x)^\bot=G_-(x)+ G_+(x)$.  

\begin{nota}
We introduce the two notations:  $E(x)=G_-(x)+G_+(x)$ and    $R(x)= G_-(x)\cap G_+(x)$. We denote the reduced space: $F(x)=E(x)/R(x)$  by $F(x)$ and we denote   the canonical projection $p~: E\rightarrow F$ by $p$.  As $G_-$ and $G_+$ are invariant by the linearized dynamics $Df$, we may define a reduced cocycle $M ~: F\rightarrow F$. But $M $ is not continuous, because $G_-$ and $G_+$ don't vary continuously.\\
Moreover, we introduce the notation: $\Vc(x)=V(x)\cap E(x)$ is the trace of the linearized vertical on $E(x)$ and $v(x)=p(\Vc(x))$ is the projection of $\Vc (x)$ on $F(x)$. We introduce a notation for the images of the reduced vertical $v(x)$ by $M^k$: $g_k(x)=M^kv(f^{-k}x)$.\\
Of course, we define an order on the set of the Lagrangian subspaces of $F(x)$ that are transverse to $v(x)$ exactly as this was done in the non-reduced case.
\end{nota}

The subspace $E(x)$ of $T_x\A_n$ is  co-isotropic with $E(x)^{\bot}=R(x)$. Hence $F(x)$ is nothing else than the symplectic space that is obtained by symplectic reduction of $E(x)$. We denote its symplectic form by $\Omega$. Then we have: $\forall (v, w)\in E(x)^2, \Omega (p(v), p(w))=\omega (v,w)$. Moreover,  $M$ is a symplectic cocycle.\\
We can notice, too, that $\dim E(x)=\dim (G_-(x)+G_+(x))=\dim G_-(x)+\dim G_+(x)-\dim (G_-(x)\cap G_+(x))=2n-p$ and deduce that $\dim F(x)=\dim E(x)-\dim (G_-(x)\cap G_+(x))=2(n-p)$.

\begin{nota} If $L$ is any Lagrangian subspace of $T_x\A_n$, we denote $(L\cap E(x))+R(x)$ by $\tilde{L}$ and  $p(\tilde{L})$ by $l$. 
\end{nota}
\begin{lemma}\label{bernard1} If $L\subset T_x\A_n$ is Lagrangian, then $\tilde{L}$ is also Lagrangian and $l=p(\tilde L)=p(L\cap E(x))$ is a Lagrangian subspace of $F(x)$. Moreover, $p^{-1}(l)=\tilde{L}$ . In particular, $v(x)$ is a Lagrangian subspace of $F(x)$ and $p^{-1}(v(x))=\Vc(x)+R(x)$.

\end{lemma}
The proof is given in \cite{Arn2}.
\begin{lemma}\label{L2}
The subspace $v(x)$ is a Lagrangian subspace of $F(x)$. Moreover, for every $k\not=0$, $g_k(f^kx)=M^kv(x)$ is transverse to $v(f^k(x))$
\end{lemma}
\demo
The first result is contained in lemma \ref{bernard1}.\\
  Let us  consider $k\not=0$ and let us assume that $ M^kv(x) \cap v(f^kx)\not=\{ 0\}$.  We may assume that $k>0$ (or we replace $x$ by $f^k(x)$ and $k$ by $-k$).
  
  Then there exists $v\in \Vc(x)\backslash \{  0\}$ such that $Df^k(x)v\in \Vc (f^k x)+(G_-(f^kx)\cap G_+(f^kx))$.  Let us write $Df^k(x)v=w+g$ with $w\in \Vc (f^kx)$ and $g\in R(f^kx)$.  We know that the orbit has no conjugate vectors (because the measure is minimizing); hence $g\not=0$. 
  
  Moreover, we know that $Df^k V(x)$ is strictly above $G_-(f^kx)$, i.e. that:
  $$\forall h\in G_-(f^kx), \forall h'\in V(f^kx),    h+h'\in (Df^k V(x))\backslash \{ 0\}\Rightarrow  \omega  (h, h+h')> 0.$$
  We deduce that: $\omega (g, w+g)>0$. 
  
 This contradicts: $Df^kv\in E(f^kx)=\left(G_+(f^kx)\cap G_-(f^kx)\right)^\bot\subset (\R g)^\bot$.
\enddemo
\begin{lemma}\label{bernard2} Let $L_1$, $L_2$ be two Lagrangian subspaces of $T_x\A_n$ transverse  to $V(x)$ such that at least one of them is contained in $E(x)$. Then, if $L_1<L_2$ (resp. $L_1\leq L_2$), we have: $l_1$ and $l_2$ are transverse to $v(x)$ and $l_1 < l_2$ (resp. $l_1\leq l_2$). We deduce that $p(G_-)<p(G_+)$.

\end{lemma}
The proof is given in \cite{Arn2}.
 \begin{lemma}\label{linegred}
 If $\mu$ is a minimizing measure, for every $x\in{\rm supp}\mu$, for all $0<k<m$, we have:
 $$g_{-k}(x)<g_{-m}(x)<p(G_-)<p(G_+)<g_m(x)<g_k(x).$$
 \end{lemma}
  \demo
  We cannot use the proof given in \cite{Arn2} that use in a crucial way the continuous dependence on time.  Let us prove by iteration on $k\geq 1$ that $p(G_+)<g_{k+1}<g_k$, i.e. that $g_{k+1}\in\Pc(p(G_+), g_{k})$ with the notations of the appendix.\\
 Because $G_+<G_1$ and because of lemma \ref{bernard2}, we have $p(G_+)<g_1$ i.e. $g_1\in\Pc(p(G_+),v)$. Taking the image by $M$, we deduce: $g_2\in \Pc(p(G_+), g_1)$. We deduce from proposition \ref{proappbis} (see the appendix) that $p(G_+)<g_2<g_1$.   The result for $g_k$ with $k\geq 1$ is just an iteration, and the result for $k\leq -1$ is very similar. \enddemo
 
 \begin{lemma}
 We have: $\displaystyle{\lim_{k\rightarrow +\infty}g_k=p(G_+)\quad{\rm and} \quad \lim_{k\rightarrow +\infty}g_{-k}=p(G_-)}$
\end{lemma}

\demo
From lemma \ref{linegred}, we deduce that the $(g_k)_{k\geq 1}$ converges to $g_+\geq p(G_+)$ and that $(g_{-k})_{k\geq 1}$ converges to $g_-\leq p(G_-)$.\\
Let us assume for example that $g_+\not=p(G_+)$. Then $W=p^{-1}(g_+)$ is transverse to $V$ and invariant by $Df$.\\
Moreover, for every $w\in W$ and $v\in G_+$, we have: $\omega(w, v)= \Omega(p(w), p(v))$. We deduce that $G_+\leq W$. We  choose a Lagrangian subspace $L$ of $T_x(\A_n)$ such that $W<L$ and $G_1<L$.\\
Because $G_1<L$, we have $L\in \Pc(G_1, V)$ and then for every $k\geq 1$: $Df^k(L\circ f^{-k} )\in \Pc(G_{k+1}, G_k)$, hence, by proposition \ref{proappbis}, $G_{k+1}<Df^kL<G_k$. Note that this implies that $G_+<Df^kL$.\\
Because $G_+\leq W<L$, we have $W\in \overline{ \Pc(G_+, L)}$ by proposition \ref{proappter} and then $W=Df^kW\circ f^{-k}\in \overline{\Pc (G_+, Df^kL\circ f^{-k})}$, and then  $G_+\leq W\leq Df^kL\circ f^{-k}$ by proposition \ref{proappbis}.\\
We have finally proved
$$\forall k\geq 1, G_+\leq W\leq Df^kL\circ f^{-k}<G_k.$$
Taking the limit, we obtain: $W=G_+$.
\enddemo
\begin{defin}
The two Lagrangian subbundles $g_-=p(G_-)$ and $g_+=p(G_+)$ are the two {\em reduced Green bundles}.
\end{defin}
\subsection{Weak hyperbolicity of the reduced cocycle}\label{sshyp}
With the notations of subsection \ref{fred}, we will now explain why the reduced cocycle is weakly hyperbolic and why $\mu$ has exactly $2p$ zero Lyapunov exponents. The proof is very similar to the one given in \cite{Arn2} for the Tonelli Hamiltonian flows, we just translate it to the discrete case.

 %We have to  be careful because the bundles that we consider are not continuous and, as this is noted in \cite{Arn2},  we don't use a continuous change of coordinates, but just  a bounded one when we say that $G_-$ or $G_+$ is the horizontal subspace (the matrix $P$ that is necessary to change the coordinates is uniformly bounded, as $P^{-1}$). \\
 We choose at every point $x\in{\rm supp}\mu$ some (linear) symplectic coordinates $(Q,P)$ of $F(x)$ such that $v(x)$ has for equation: $Q=0$ and $g_+(x)$ has for equation $P=0$. We will be more precise on this choice later. Then the matrix of $M^k(x)=M(f^{k-1}(x))\dots M(x)$ in these coordinates is a symplectic matrix: $M^k(x)=\begin{pmatrix}
 a_k(x)&b_k(x)\\
0&d_k(x)\\
 \end{pmatrix}$. As $M^k(x)v(x)=g_k(f^kx)$ is a Lagrangian subspace of $E(f^kx)$ that is transverse to the vertical,  then $\det b_k(x)\not=0$ and there exists a symmetric matrix $s_k^+(f^kx)$ whose graph is $g_k(f^kx)$, i.e: $d_k(x)=s_k^+(f^k(x))b_k(x)$. Moreover, the family $(s^+_k(x))_{k>0}$ being decreasing and tending to zero (because by hypothesis the horizontal is $g_+$), the symmetric matrix $s^+_k(f^kx)$ is positive definite. Moreover, the matrix $M^k(x)$ being symplectic, we have: 
 $$\left(M^k(x)\right)^{-1}=\begin{pmatrix} {}^td_k(x)&-{}^tb_k(x)\\
 0&{}^ta_k(x)\\
\end{pmatrix}$$
and by definition of $g_{-k}(x)$, if it is the graph of the matrix $s^-_k(x)$ (that is negative definite), then: ${}^ta_k(x)=-s_k^-(x){}^tb_k(x)$ and finally:
$$M^k(x)=\begin{pmatrix}
-b_k(x)s_k^-(x)& b_k(x)\\
0& s_k^+(f^kx)b_k(x)\\
\end{pmatrix}
$$
Let us be now more precise in the way we choose our coordinates; as explained at the end of the introduction, we may associate  an almost complex structure $J$ and then a Riemannian metric $(.,.)_x$ defined by: $(v,u)_x=\omega (x)(v, Ju)$ with the symplectic form $\omega$ of $\A_n$; from now on,  we work with this fixed Riemannian metric of $\A_n$. We choose on $ G_+(x)=p^{-1}(g_+(x))$ an orthonormal basis whose last vectors are in $R(x)$ and complete it in a symplectic basis whose last vectors are in $V(x)$. We denote  the associated  coordinates of $T_x\A_n$ by $(q_1, \dots, q_n, p_1, \dots, p_n)$. These (linear) coordinates don't depend in a continuous way on the point $x$ (because $G_+$ doesn't), but in a bounded way. Then $G_-(x)=p^{-1}(g_-(x))$ is the graph of a symmetric matrix whose kernel is $R(x)$ and then on $G_-(x),$ we have: $p_{n-p+1}=\dots =p_n=0$. An element of $E(x)$ has coordinates such that $p_{n-p+1}=\dots =p_n=0$, and an element of $F(x)=E(x)/R(x)$ may be identified with an element with coordinates $(q_1, \dots , q_{n-p}, 0, \dots , 0, p_1, \dots , p_{n-p}, 0, \dots  , 0)$. We then use on $F(x)$ the norm $\displaystyle{\sum_{i=1}^{n-p}(q_i^2+p_i^2)}$, which is the norm for the Riemannian metric   of the considered element of $F(x)$. Then this norm depends in a measurable way on $x$.

\begin{lemma}\label{LJ}
For every $\varepsilon >0$, there exists a measurable subset $J_\varepsilon$ of ${\rm supp}\mu$ such that:
\begin{enumerate}
\item[$\bullet$] $\mu (J_\varepsilon)\geq 1-\varepsilon$;
\item[$\bullet$] on $J_\varepsilon$, $(s_k^+)$ and $(s_k^-)$ converge uniformly ;
\item[$\bullet$] there exists two constants $\beta=\beta(\varepsilon)>\alpha=\alpha(\varepsilon)>0$ such that: $\forall x\in J_\varepsilon,\beta{\bf 1}\geq  -s_-(x)\geq \alpha {\bf 1}$ where $g_-$ is the graph of $s_-$.
\end{enumerate}
\end{lemma}
\demo  This is a consequence of Egorov theorem and of the fact that $\mu$-almost everywhere on ${\rm supp}\mu$, $g_+$ and $g_-$ are transverse and then $-s_-$ is positive definite.
\enddemo
 We deduce:
 \begin{lemma}\label{LCVU}
 Let $J_\varepsilon$ be as in the previous lemma. On the set $\{ (k,x)\in\N\times J_\varepsilon, f^k(x)\in J_\varepsilon\}$, the sequence of conorms $(m(b_k(x)))$ converge uniformly to $+\infty$, where $m(b_k)=\| b_k^{-1}\| ^{-1}$.
 \end{lemma}
 \demo Let $k, x$ be as in the lemma.\\
 The matrix  $M_k(x)=\begin{pmatrix}
 -b_k(x)s_k^-(x)& b_k(x)\\
 0& s_k^+(f^kx)b_k(x)\\
 \end{pmatrix}
 $ being symplectic, we have: \\
 $-s_k^-(x){}^tb_k(x)s_k^+(f^kx)b_k(x)={\bf 1}$ and thus  
 $-b_k(x)s_k^-(x){}^tb_k(x)s_k^+(f^kx)={\bf 1}$ and:\\
  $b_k(x)s_k^-(x){}^tb_k(x)=-\left(s_k^+(f^kx)\right)^{-1}$. \\
 We know that on $J_\varepsilon$, $(s_k^+)$ converges uniformly to zero. Hence,  for every $\delta>0$, there exists $N=N(\delta) $ such that: $k\geq N\Rightarrow \| s_k^+(f^kx)\|\leq \delta$. Moreover, we know that $\| s_k^-(x)\|\leq \beta$. Hence, if we choose $\delta'=\frac{\delta^2}{\beta}$, for every $k\geq N=N(\delta')$ and $x\in J_\varepsilon$ such that $f^kx\in J_\varepsilon$, we obtain: 
 $$\forall v\in\R^p,\beta \| {}^tb_k(x)v\|^2= {}^tv b_k(x)(\beta{\bf 1}){}^tb_k(x)v\geq - {}^tv b_k(x)s_k^-(x){}^tb_k(x)v={}^tv\left(s_k^+(f^kx)\right)^{-1}v$$
 and we have: ${}^tv\left(s_k^+(f^kx)\right)^{-1}v\geq \frac{\beta}{\delta^2}\| v\|^2$ because $s_k^+(f^kx)$ is a positive definite matrix that is less than $\frac{\delta^2}{\beta}{\bf 1}$. We finally obtain: $\| {}^tb_k(x)v\|\geq \frac{1}{\delta}\| v\|$ and then the result that we wanted.
 \enddemo
 From now we fix  a small constant $\varepsilon>0$, associate  a set $J_\varepsilon$ with $\varepsilon$  via lemma \ref{LJ} and two constants $0<\alpha<\beta$; then  there exists $N\geq 0$ such that
 $$\forall x\in J_\varepsilon, \forall k\geq N, f^k(x)\in J_\varepsilon\Rightarrow m(b_k(x))\geq \frac{2}{\alpha}.$$

 \begin{lemma}
 Let $J_\varepsilon$ be as in lemma \ref{LJ}. For $\mu$-almost point $x$ in $J_\varepsilon$, there exists a sequence of integers $(j_k)=(j_k(x))$ tending to $+\infty$ such that: 
 $$\forall k\in \N, m(b_{j_k}(x)s^-_{j_k}(x))\geq \left( 2^\frac{1-\varepsilon}{2N}\right)^{j_k}.$$
 \end{lemma} 
 \demo
 As $\mu$ is ergodic for $f$, we deduce from Birkhoff ergodic theorem that for almost every point $x\in J_\varepsilon$, we have:
 $$\lim_{\ell\rightarrow +\infty}\frac{1}{\ell}\sharp \{ 0 \leq k\leq \ell-1; f^k(x)\in J_\varepsilon\}=\mu (J_\varepsilon)\geq 1-\varepsilon.$$
 We introduce the notation: $N(\ell)=\sharp \{ 0 \leq k\leq \ell-1; f^k(x)\in J_\varepsilon\}$.\\
 For such an $x$ and every $\ell\in\N$, we find a number $n(\ell)$ of integers:
 $$0=k_1\leq k_1+N\leq k_2\leq k_2+N\leq k_3\leq k_3+N\leq \dots \leq k_{n(\ell)}\leq \ell$$
 such that $f^{k_i}(x)\in J_\varepsilon$ and $n(\ell)\geq [\frac{N(\ell)}{N}]\geq \frac{N(\ell)}{N}-1$. In particular, we have: $\frac{n(\ell)}{\ell}\geq\frac{1}{N}(\frac{N(\ell)}{\ell}-\frac{N}{\ell})$, the right term converging to $\frac{\mu (J_\varepsilon)}{N}\geq \frac{1-\varepsilon}{N}$ when $\ell$ tends to $+\infty$. Hence, for $\ell$ large enough, we find: $n(\ell)\geq 1+  \ell \frac{1-\varepsilon}{2N}$.\\ 
As $f^{k_i}(x)\in J_\varepsilon$ and $k_{i+1}-k_i\geq N$, we have:  $m(b_{k_{i+1}-k_i}(f^{k_i}(x)))\geq \frac{2}{\alpha}$. Moreover, we have: $m(s_{k_{i+1}-k_i}^-(f^{k_i}x))\geq \alpha$; hence: 
$$m(b_{k_{i+1}-k_i}(f^{k_i}x)s_{k_{i+1}-k_i}^-(f^{k_i}x))\geq 2.$$
But the matrix $-b_{k_{n(\ell)}}(x)s^-_{k(n(\ell))}(x)$ is the product of $n(\ell)-1$
 such matrix. Hence:
 $$m(b_{k_{n(\ell)}}(x)s^-_{k(n(\ell))}(x))\geq 2^{n(\ell)-1}\geq 2^{\ell\frac{1-\varepsilon}{2N}}\geq \left( 2^\frac{1-\varepsilon}{2N}\right)^{k_{n(\ell)}}.
 $$
 
 \enddemo
Let us now come back to the whole tangent space $T_x\A_n$ with a slight  change in the coordinates that we use. We defined the symplectic coordinates $(q_1, \dots , q_n, p_1, \dots , q_n)$ and now we use the non symplectic ones: \\
$(Q_1, \dots, Q_n,P_1, \dots , P_n)=(q_{n-p+1}, \dots, q_n, q_1, \dots, q_{n-p}, p_1, \dots , p_n)$.  Then:
\begin{enumerate}
 \item[$\bullet$] $(Q_1, \dots , Q_p)$ are coordinates in $R(x)$;
 \item[$\bullet$] $(Q_1, \dots , Q_n)$ are coordinates in $G_+(x)$;
 \item[$\bullet$] $(Q_1, \dots , Q_n, P_{1}, \dots , P_{n-p})$ are coordinates of $E(x)=G_+(x)+G_-(x)$.
 \end{enumerate}
We write then the matrix of $Df^k(x)$  in these  coordinates $(Q_1, \dots , Q_n, P_1, \dots   , P_n)$ (which are not symplectic):
$$\begin{pmatrix}
A^1_k(x)&A^2_k(x)&A^3_k(x)&A^4_k(x)\\
0&b_k(x)s_k^-(x)&b_k(x)&A^5_k(x)\\
0&0& s_k^+(f^kx)b_k(x)&A^6_k(x)\\
0&0&0&A^7_k(x)\\
\end{pmatrix}$$
where the blocks correspond  to the decomposition $T_x\A_n=E_1(x)\oplus E_2(x)\oplus E_3(x)\oplus E_4(x)$ with $\dim E_1(x)=\dim E_4(x)=p$ and $\dim E_2(x)=\dim E_3(x)=n-p$.\\
We have noticed that $E_1(x)=E(x)\subset E^c(x)$ and that $G_+(x)=E_1(x)\oplus E_2(x)$.\\
If $x\in J_\varepsilon$, we have found a sequence $(j_k)$ of integers tending to $+\infty$ so that: 
 $$\forall k\in \N, m(b_{j_k}(x)s^-_{j_k}(x))\geq \left( 2^\frac{1-\varepsilon}{2N}\right)^{j_k}.$$
 We deduce:
 $$\forall v\in E_2(x)\backslash \{ 0\}, \frac{1}{j_k}\log\left( \| b_{j_k}(x)s^-_{j_k}(x)v\|\right)\geq  \frac{1-\varepsilon}{2N}\log 2 +\frac{\|v\|}{j_k};$$
 and because $E_1(x)\subset E^c(x)$: 
 $$\forall v\in G_+(x)\backslash E_1(x), \liminf_{k\rightarrow +\infty}\frac{1}{k}\log \| Df^k(x)v\|\geq \frac{1-\varepsilon}{2N}\log 2.$$
Hence there are at least $n-p$  Lyapunov exponents  bigger than $ \frac{1-\varepsilon}{2N}\log 2$ and then bigger than $0$ for the linearized dynamics. Because this dynamics is symplectic, we deduce that it has at least $n-p$ negative Lyapunov exponents (see \cite{BocVia1} ). As we noticed that the linearized flow has at least $2p$ zero Lyapunov exponents, we deduce that $\mu$ has   exactly  $n-p$ positive Lyapunov  exponents, exactly  $n-p$ negative Lyapunov exponents and  exactly $2p$ zero Lyapunov exponents.\\
This finishes the proof of theorem \ref{nbexp}.

\remk   Let us notice that we proved too that for $\mu$ almost every $x\in {\rm supp} \mu$, we have:  $E^u(x)\subset G_+(x)$, and then $G_+(x)=E^u(x)\oplus R(x) $.

\section{Size of the Lyapunov exponents and angle between the two Green bundles}\label{sAngles}
The idea to prove theorem \ref{sizeexp} is to use the reduced Green bundles that we introduced just before and to adapt the proof that we gave  in \cite{Arn4} in the case of weak hyperbolicity. We use the same notations as in section \ref{sGreen}.\\

The Lagrangian bundles $g_-$ and $g_+$ being transverse to the vertical at   every point of ${\rm supp}\mu$, there exist two symmetric matrices $\esse$ and $\U$ such that $g_-$ (resp. $g_+$) is the graph of $\esse$ (resp. $\U$) in the coordinates$(q_1, \dots , q_{n-p},   p_1, \dots , p_{n-p})$ that we defined at the beginning of subsection \ref{sshyp}. We denote by $(e_1, \dots , e_{2(n-p)})$ the associated symplectic basis. 
 As $g_-$ and $g_+$ are transverse $\mu$-almost everywhere,  we know that  there exists $\varepsilon>0$ such that $A_\varepsilon=\{ x\in {\rm supp}\mu; \U-\esse\geq \varepsilon {\bf 1}\}$ has positive $\mu$-measure. We use the notation $x_k=f^k(x)$.We may then assume that $x_0\in A_\varepsilon$ and that $\{ k\geq 0; \U (x_k)-\esse(x_k)>\varepsilon{\bf 1}\}$ is infinite. Let us notice that in this case, $g_-$ and $g_+$ are transverse along the whole orbit of $x_0$ (but $\U-\esse$ can be very small at some points of this orbit).  Let us note too that in fact $\U=0$. \\
Hence, for every $k\in\N$, there exists a unique positive definite matrix $S_0(x_k)$ such that: $S_0(x_k)^2=\U(x_k)-\esse (x_k)$. Let us recall that a matrix $M=\begin{pmatrix} a & b\\ c&d\\ \end{pmatrix}$ of dimension $2(n-p)$ is symplectic if and only if its entries satisfy the following equalities:
$${}^tac={}^tca;\quad {}^tbd={}^tdb;\quad {}^tda-{}^tbc={\bf 1}.$$
We define along the orbit of $x_0$ the following change of basis: 
$P=\begin{pmatrix} S_0^{-1}& S_0^{-1}\\ \esse S_0^{-1}& \U S_0^{-1}\\ \end{pmatrix}.$
Then it defines a symplectic change of coordinates, whose inverse is: 
$$Q=P^{-1}=\begin{pmatrix} 0&{\bf 1}\\ -{\bf 1}& 0\\ \end{pmatrix} {}^tP\begin{pmatrix} 0&-{\bf 1}\\ {\bf 1}& 0\\ \end{pmatrix}=\begin{pmatrix} S_0^{-1}\U & -S_0^{-1}\\ -S_0^{-1}\esse & S_0^{-1}\\ \end{pmatrix}.$$
We use this symplectic change of coordinates along the whole orbit of $x_0$. More precisely, if we denote the matrix of $M^k$ in the usual canonical basis $e=(e_i)$  by $M_k$,  then the matrix of $M^k$ in the basis $Pe=(Pe_i)$ is denoted by $\tilde M_k$; we have then: $\tilde M_k(x_h)= P^{-1}(x_{h+k})M_k(x_h)P(x_h)$. Note that the image of the horizontal (resp. vertical) Lagrangian plane by $P$ is $g_-$ (resp. $g_+$). As the bundles $g_-$ and $g_+$ are invariant by $M$, we deduce that $\tilde M_k=\begin{pmatrix} \tilde a_k&0 \\
 0 &\tilde d_k\\ \end{pmatrix}$; we have  ${}^t\tilde a_k \tilde d_k={\bf 1}$ because this matrix is symplectic.\\
 Moreover, we know that: $M_k(x_h)=\begin{pmatrix} -b_k(x_h)s_{-k}(x_h)& b_k(x_h)\\
 c_k(x_h)& s_k(x_{k+h})b_k(x_h)\\ \end{pmatrix}$ where $g_k(x_h)=M^k.v(x_{h-k})$ is the graph of $s_k(x_h)$.\\
 Writing that  $\tilde M_k(x_h)=\begin{pmatrix} \tilde a_k(x_h)&0 \\
 0 &\tilde d_k(x_h)\\ \end{pmatrix}= P^{-1}(x_{h+k})M_k(x_h)P(x_h)$, we obtain firstly: 
 $$\begin{matrix} S_0(x_{h+k})^{-1}&{}^tb_k(x_h)S_0(x_h)^{-1}=\hfill\\ &S_0(x_{h+k})^{-1}(\esse (x_{h+k})-s_{k}(x_{h+k}))b_k(x_h)(s_{-k}(x_h)-\esse (x_h))S_0(x_h)^{-1};
 \end{matrix}$$
 $$\begin{matrix}-S_0(x_{h+k})^{-1}&{}^tb_k(x_h)S_0(x_h)^{-1}\hfill\\
 &=S_0(x_{h+k})^{-1}(\U (x_{h+k})-s_{k}(x_{h+k}))b_k(x_h)(\U (x_h)-s_{-k}(x_h))S_0(x_h)^{-1}.
 \end{matrix}$$
  We deduce that:
 $\tilde a_k(x_h)=S_0(x_{h+k})b_k(x_h)(\esse (x_h)-s_{-k}(x_h))S_0(x_h)^{-1}$  and: \\
  $\tilde d_k(x_h)=S_0(x_{h+k})b_k(x_h)(\U (x_h)-s_{-k}(x_h))S_0(x_h)^{-1}$.
  
 Because of the changes of basis that we used, $(\tilde a_k(x_h))_k$ represents the linearized dynamics $(M^k_{|g_-(x_h)})_k$ restricted to $g_-$ and $(\tilde d_k(x_h))_k$ the linearized dynamics restricted to $g_+$.  Hence we need to study $(\tilde d_k(x_h))$ to obtain some information about the positive Lyapunov exponents of $\mu$. Let us compute:\\
 ${}^t\tilde d_k (x_h)=\tilde a_k(x_h)^{-1}= S_0(x_h)(\esse (x_h)-s_{-k}(x_h))^{-1}b_k(x_h)^{-1}S_0(x_{h+k})^{-1}$; we deduce:\\
 $${}^t\tilde d_k(x_h)\tilde d_k(x_h)=S_0(x_h)(\esse (x_h)-s_{-k}(x_h))^{-1}(\U (x_h)-s_{-k}(x_h))S_0(x_h)^{-1}$$
 $$=S_0(x_h)(\esse (x_h)-s_{-k}(x_h))^{-1}(\U (x_h)-\esse (x_h)+\esse (x_h)-s_{-k}(x_h))S_0(x_h)^{-1}$$
 $$={\bf 1}+ S_0(x_h)(\esse (x_h)-s_{-k}(x_h))^{-1} S_0(x_h)$$
 $$={\bf 1}+(\U (x_h)-\esse (x_h))^\frac{1}{2}(\esse (x_h)-s_{-k}(x_h))^{-1} (\U (x_h)-\esse (x_h))^\frac{1}{2}.$$
 Let us denote the conorm of $a$ (for the usual Euclidean norm of $\R^{n-p}$) by: $m(a)=\| a^{-1}\|^{-1}$. Then we have:
 $$m(\tilde d_k(x_h))^2=m({}^t\tilde d_k(x_h)\tilde d_k(x_h));$$
 Let us recall that on ${\rm supp}\mu$,    $G_+$ is uniformly far from the vertical.  This implies that $S_1-S_{-1}$ is uniformly bounded on the (compact) support of $\mu$ (see the notations before theorem \ref{sizeexp} for the definition of $S_1$ and $S_{-1}$). Then their restriction to $q_{n-p+1}=\dots =q_n=0$ is uniformly bounded too (by the same constant);  let $C$ designate  $\sup \| s_1-s_{-1}\|$ above the  support of $\mu$. We have then:  $m(\tilde d_k(x_h))^2\geq 1+\frac{1}{C}m((\U-\esse )(x_h))$; indeed, we know that: $s_1-s_{-1}\geq \esse -s_{-k} >0$.
 
 The entry $\tilde d_k$ being multiplicative, we deduce that:
 $$m(\tilde d_k(x_0))^2\geq \prod_{n=0}^{k-1} (1+\frac{1}{C}m(\U (x_h)-\esse (x_h)))$$
 and: 
 $$\frac{1}{k}\log m(\tilde d_k(x_0))\geq \frac{1}{2k}\sum_{n=0}^{k-1}\log (1+\frac{1}{C}m(\U (x_h)-\esse (x_h))).$$
When $k$ tends to $+\infty$, we deduce from Birkhoff's ergodic theorem that: 
$$(*)\quad \liminf_{k\rightarrow \infty}\frac{1}{k}\log m(\tilde d_k(x_0))\geq\frac{1}{2}\int \log\left( 1+\frac{1}{C}m(\U (x)-\esse (x))\right) d\mu (x).$$
Let us recall that $(\tilde d_k(x_0))$ represents the dynamics along $g_+$, but the change of basis that we have done is not necessarily bounded. To obtain a true information about the Lyapunov positive exponents of $(M^k)$ , we need to have a result for the matrix $D_k$ of $(M^k_{|g_+(x_0)})$ in the basis $(e_1, \dots , e_{n-p})$ of $g_+$ whose matrix in the usual coordinates is: $\begin{pmatrix} {\bf 1}\\ \U \\ \end{pmatrix}=\begin{pmatrix} {\bf 1}\\ {\bf 0} \\ \end{pmatrix}$. Since $(\tilde d_k)$ is the matrix of $M^k$ in the basis whose matrix is $\begin{pmatrix} S_0^{-1}\\ \U S_0^{-1} \\ \end{pmatrix}=\begin{pmatrix} S_0^{-1}\\ {\bf 0} \\ \end{pmatrix}$, we deduce that:  $D_k(x_0)=S_0(x_k)\tilde d_k(x_0)S_0(x_0)^{-1}$ and:\\
 $m(D_k(x_0))\geq m(S_0(x_k)) m(\tilde d_k(x_0)) m(S_0(x_0)^{-1})=\left( m(\U(x_k)-\esse (x_k))\right)^\frac{1}{2} m(\tilde d_k(x_0)) m(S(x_0)^{-1})$. \\
 We have $(*)$ and we know that: $\displaystyle{\liminf_{k\rightarrow \infty} m(\U (x_k)-\esse (x_k))\geq \varepsilon}$. We deduce: 
$$\lambda (\mu)\geq  \liminf_{k\rightarrow \infty}\frac{1}{k}\log m(D_k(x_0))\geq\frac{1}{2}\int \log\left( 1+\frac{1}{C}m(\U (x)-\esse (x))\right) d\mu (x).$$

Because $\Delta S$ is a symmetric positive semi-definite matrix, we have: $q_+(\Delta S)=q_+(\Delta S_{|(\ker \Delta S)^\bot})$. If we look at the definition of the coordinates $(q_i, p_i)$, we note that: $\Delta S_{|(\ker \Delta S)^\bot}=-\esse=\U-\esse$.  Hence we have proved theorem \ref{sizeexp}.

  \section{Shape of the support of the minimizing measures and Lyapunov exponents}\label{sKAM}
  
  \subsection{Some reminders about discrete weak KAM theory}\label{ssweakkam}
  The general reference for what is contained in this section is the article of Garibaldi \& Thieullen  \cite{GarThi} and the results that they obtain are very similar to the ones obtained by A.~Fathi in the setting of the time-continuous weak K.A.M. theory (see \cite{Fat1}). The dynamics that we study here are contained in the ones that they study and that are called ``ferromagnetic''. In \cite{GarThi}, a big part of the article deals with a Lagrangian function $L:\R^n\times \R^n\rightarrow \R$ that is defined by $L(x, v)=S(x, x+v)$ (let us recall that $S$ is a generating function for $F$) and the action $\Fc$ is denoted by $\Lc$ by them.
  They prove the existence of a unique $\overline \Lc\in \R$ such that there exists two $\Z^n$-periodic continuous functions $u_-, u_+:\R^n\rightarrow \R$ such that:
  $$\forall x\in\R^n, u_-(y)=\inf_{x\in\R^n} u_-(x)+S(x,y)-\overline\Lc\quad{\rm and}\quad u_+(x)=\sup_{y\in\R^n} u_+(y)-S(x,y)+\overline\Lc$$
  and that the infimum (resp. supremum) is attained at some point.  
    \begin{propo}{\em (Garibaldi-Thieullen)}  \\
  With the above notations and assumptions:
  $$\bar \Lc=\inf _\mu \int_{\R^n\times\R^n} S(x,y)d\tilde\mu(x,y);$$
  where the infimum is taken on the set of the Borel probability measures that are invariant by $f$ and $\tilde \mu$ is any lift of $\mu$ to a fundamental domain of $\R^n\times \R^n$ for the projection $(x,y)\mapsto (x, -\frac{\partial S}{\partial x}(x,y))$ onto $\T^n\times \R^n$. Moreover the infimum is attained for some invariant $\mu$.
  
  \end{propo}
  Then such a measure $\mu$ where the infimum is attained  is a minimizing measure, but all the minimizing measures are not like that.  We define:
  
  \begin{defin}
  A configuration $(x_k)_{k\in\Z}$ of points of $\R^n$  is {\em strongly minimizing} if for any pairs $m<\ell$ et $m'<\ell'$ and any configuration $(y_k)_{k\in\Z}$ satisfying $y_{m'}-x_m\in\Z^n$ and $y_{\ell'}-x_\ell\in\Z^n$, we have:
  $$\bar \Fc (x_m, x_{m+1}, \dots, x_\ell)\leq \bar \Fc(y_{m'}, \dots, y_{\ell'}).$$
  The corresponding orbit for $f$ is then {\em strongly minimizing}.
  \end{defin}
  It is not hard to see that if $\mu$ is a Borel probability measure invariant by $f$ then it satisfies the equality in the proposition above if and only if its support is filled by strongly minimizing orbits.
  
  \begin{nota}
  The union of the supports of all the measures $\tilde\mu$  where $\mu$ is strongly minimizing   is called the {\em Mather set} and is denoted by $\Mc(S)$.\end{nota}
  
  \begin{nota}
  We introduce the notations $\bar S(x,y)=S(x,y)-\bar\Lc$ and $$\bar\Fc(x_1, \dots x_m)=\Fc(x_1, \dots, x_m)-(m-1)\bar\Lc=\sum_{i=1}^{m-1}\bar S(x_i, x_{i+1}).$$
  \end{nota}
    
  From now on, we will call $\bar\Fc$ the action and we will consider minimizing orbits for this action (in fact minimizing orbits are the same for the two actions).
  
  \begin{defin} Let $u:\R^n\rightarrow \R$ be a $\Z^n$-periodic and continuous function.  Then
  \begin{enumerate}
  \item $u$ is a {\em subaction} with respect to $S$ if:
  $$\forall x, y\in\R^n, u(y)-u(x)\leq \bar S(x,y);$$
  \item $u$ is {\em backward calibrated} if it is a subaction and
  $$\forall y\in\R^n, u(y)=\inf_{x\in\R^n}(u(x)+\bar S(x, y));$$
  \item $u$ is {\em forward calibrated} if it is a subaction and
  $$\forall x\in\R^n, u(x)=\sup_{y\in\R^n}(u(y)-\bar S(x, y)).$$
 \end{enumerate}
  
  \end{defin}   
  \begin{defin}
Let $K\geq 0$ be  a constant.  A function $u:\R^n\rightarrow \R$ is $K$-{\em semiconcave}  if for every $x_0\in\R^N$, there exists $p_{x_0}\in\R^n$ (that is non-necessarily unique) such that:
$$\forall x\in \R^n, \| x-x_0\|\leq 1\Rightarrow   u(x)\leq u(x_0)+p_{x_0}(x-x_0)+K\| x-x_0\|^2.$$
Then $p_{x_0}$ is a {\em superdifferential} for $u$ at $x_0$.
The function $u$ is $K$-{\em semiconvex} if $-u$ is semiconcave. 
  \end{defin}
  Let us recall some well-known properties of semiconcave functions (see for example \cite{CanSin}); we assume that $u$ is $K$-semiconcave.
  \begin{enumerate}
  \item[$\bullet$] if $x_0$ is a local minimizer for $u$, then $u$ is differentiable at $x_0$;
  \item[$\bullet$] a infimum of $K$-semiconcave functions is $K$-semiconcave;
  \item[$\bullet$] every semiconcave function is Lipschitz.
  \end{enumerate}  
  A consequence of these properties is that any backward calibrated subaction is semiconcave and any forward calibrated subaction  is semiconvex.
  
  \begin{nota}
  If $u:\R^n\times \R^n\rightarrow \R$ is a subaction, then:
  $$\Nc(u)=\{ (x,y)\in\R^n\times\R^n; u(y)=u(x)+\bar S(x,y)=u(x)+S(x,y)-\bar\Lc\}.$$
  \end{nota}
  
  \begin{remk}
  Note that for every $(x, y)\in \Nc(u)$, then $u$ is differentiable at $x$ and $du(x)=-\frac{\partial S}{\partial x}(x,y)$. Indeed, the map $ (z\mapsto u(z)+\bar S(z, y))$ is semiconcave and $x$ is a minimizer. Hence $u$ is differentiable a $x$ and $du(x)+\frac{\partial S}{\partial x}(x,y)=0$.
  \end{remk}

 Let us give a result that is very similar to a one given in \cite{Ber1} in the time-continuous case.
 
 \begin{nota}
 If $u_-:\R^n\rightarrow \R$ is a backward calibrated subaction, then for every $y\in \R^n$, we denote by $\Sigma (y)$ the set of the $x\in \R^n$ where:
 $$u_-(y)=u_-(x)+\bar S(x,y).$$
 \end{nota}
 
 \begin{prop}\label{Pbackward}
 Let $u_-:\R^n\rightarrow \R$ be a backward calibrated subaction. Then, if $y\in\R^n$ and $x\in \Sigma (y)$, $\frac{\partial S}{\partial y}(x,y)$ is a superdifferential for $u_-$ at $y$.\\
 Moreover, $u_-$ is differentiable at $y$  if and only if $\Sigma(y)=\{ x\}$ has exactly one element. Then in this case $du_-(x)=-\frac{\partial S}{\partial x}(x, y)$ and $du_-(y)=\frac{\partial S}{\partial y}(x,y)$.
 \end{prop}
 There is of course a similar statement for the forward calibrated subactions.
 \begin{proof}
 Assume that $x\in\Sigma (y)$. Then if $z\in \R^n$ satisfies $\| z-x\| \leq 1$, we have:
 $$\begin{matrix}u_-(z)&\leq u_-(x)+S(x,z) \leq u_-(x)+S(x, y)+(S(x,z)-S(x,y))\hfill\\
 &\leq 
  u_-(y)+S(x,z)-S(x,y)\leq u_-(y)+\frac{\partial S}{\partial y}(x,y)(z-y)+K\| z-y\|^2.\end{matrix}$$
  Hence $\frac{\partial S}{\partial y}(x,y)$ is a superderivative for $u_-$ at $y$.\\
  Assume that $\Sigma (y)$ has at least two elements $x_1$ and $x_2$. Then $\frac{\partial S}{\partial y}(x_1,y)$ and $\frac{\partial S}{\partial y}(x_2,y)$ are two superderivatives for $u_-$ at $y$. Because of the twist condition and because $x_1\not=x_2$, we have $\frac{\partial S}{\partial y}(x_1,y)\not= \frac{\partial S}{\partial y}(x_2,y)$. The function $u_-$ has then two superderivatives at $y$ and then is not differentiable at $y$.\\
  Assume now hat $\Sigma (y)=\{ x\}$ has exactly one element. Let $y'$ be close to $y$. Then every element $x'$ of $\Sigma (y')$ is close to $x$ and we have:
  $$\begin{matrix} u_-(y')&=u(x')+S(x', y')=u(x')+S(x', y)+(S(x', y')-S(x', y))\hfill\\&\geq u_-(y)+\frac{\partial S}{\partial y}(x', y)(y'-y)+o(\| y'-y\|)\hfill\\
  &\geq u_-(y)+\frac{\partial S}{\partial y}(x, y)(y'-y)+o(\| y'-y\|).\hfill\end{matrix}.$$
  This proves that $u$ is differentiable at $y$ and that $du(y)= \frac{\partial S}{\partial y}(x,y)$.  The fact that $du(x) =-\frac{\partial S}{\partial x}(x,y)$ is a consequence of the remarks that we made previously.
 
\end{proof}
 \begin{remk}
We deduce from  proposition \ref{Pbackward} that if a backward calibrated subaction $u_-:\R^n\rightarrow \R$ is differentiable at  $x_0$, if we use the notation $\Sigma (x_i)=\{ x_{i+1}\}$, then $u$ is differentiable at every $x_i$ and $du(x_i)=-\frac{\partial S}{\partial x}(x_i,x_{i-1})=\frac{\partial S}{\partial y}(x_{i+1}, x_{i})$, i.e. $(x_i, du(x_i))_{i\in\N}=(x_i, \frac{\partial S}{\partial y}(x_{i+1},x_i))_{i\in\N}$ is a backward orbit for $F$. Moreover, the configuration $(x_i)_{i\geq 0}$ is strongly minimizing.
 \end{remk}

 \begin{propo} {\rm (Garibaldi-Thieullen)} For any subaction $u$, we have: $\emptyset\not=\Mc(S)\subset \Nc(u)$.
  
  \end{propo}

  Moreover they prove:
  \begin{propo} {\rm (Garibaldi-Thieullen)} For any backward calibrated subaction $u_-$, there exists a forward calibrated subaction $u_+$ such that:
  \begin{enumerate}
  \item $u_-\leq u_+$;
  \item $u_{-|\Mc(S)}=u_{+|\Mc(S)}.$
  \end{enumerate}
  \end{propo}
  Such a pair $(u_-, u_+)$ will be called a pair of {\em conjugate calibrated subactions} and we introduce the notation
  
  \begin{nota}
  If $(u_-, u_+)$ is a pair of conjugate calibrated subactions, we denote by $\Ic (u_-, u_+)$ the set:
  $$\Ic(u_-, u_+)=\{ x\in\R^n; u_-(x)=u_+(x)\}.$$
  \end{nota}
  Note that $\Mc(S)\subset \Ic(u_-, u_+)$. Note too that $u_-$ and $u_+$ are differentiable above $\Ic(u_-, u_+)$ with the same derivative. We use the following notation
  $$\widetilde \Ic(u_-, u_+)=\{ (x, du_-(x)); x\in \Ic (u_-, u_+)\}.$$
\subsection{Ma\~n\'e potential and images of the vertical fiber}
In the discrete case,  an  action potential can be defined that is an analogue of the one given  by R.~Ma\~n\'e in \cite{Man1}:

\begin{defin}
Let $m\geq 1$ be an integer. The {\em action potential} $\Ac_m: \R^n\times\R^n\rightarrow \R$ is defined by:
$$\forall x, y\in\R^n, \Ac_m(x, y)=\inf\{\sum_{i=1}^m\bar S(x_{i-1}, x_i); x_0=x, x_m=y\}.$$
\end{defin}
Let us give a result that is very similar to a statement given by P.~Bernard in \cite{Ber1}. We use a notation:

\begin{nota}
If $m\geq 1$ is an integer and $x, y\in\R^n$, then $\Sigma_m(x, y)\subset (\R^n)^{m+1}$ is the set of the $(x_0, x_1, \dots, x_m)$ such that $x_0=x$, $x_m=y$ and 
$$\Ac_m(x, y)=\sum_{i=1}^m\bar S(x_{i-1}, x_i).$$
\end{nota}

\begin{prop}\label{Ppotential}
 Let $m\geq 1$ be an integer. Then $\Ac_m$ is semiconcave. Let $x, y\in\R^n$ be two points. Then   $\Sigma(x,y)\not=\emptyset$ and if $(x_0, \dots, x_m)\in \Sigma_m(x,y)$, it is the projection of a unique orbit for $F$ that is:
 $$(x_0, -\frac{\partial S}{\partial x}(x_0, x_1)), (x_1, -\frac{\partial S}{\partial x}(x_1, x_2)= \frac{\partial S}{\partial y}(x_0, x_1)), \dots, (x_m,  \frac{\partial S}{\partial y}(x_{m-1}, x_m));$$
 and  $( \frac{\partial S}{\partial x}(x_0, x_1), \frac{\partial S}{\partial y}(x_{m-1}, x_m))$ is a superdifferential for $\Ac_m$ at $(x,y)$.
 Moreover,the following assertions are equivalent:
 \begin{enumerate}
 \item[(i)] $\Ac_m$ is differentiable with respect to $x$ at $(x, y)$;
 \item[(ii)] $\Ac_m$ is differentiable with respect to $y$ at $(x, y)$;
 \item[(iii)] $\Sigma(y)=\{ (x_0, \dots , x_m)\}$ has exactly one element.
 \end{enumerate} 
 \end{prop}
 
 \begin{proof}
 The function $\Ac_m$ is the infimum of a uniformly semiconcave, bounded from below and coercive familiy. Hence it is semiconcave and the infimum is attained.\\
If $(x_0, \dots, x_m)\in \Sigma_m(x, y)$,  we have an infimum and then the partial derivatives vanish and:
 $$\frac{\partial S}{\partial y}(x_0, x_1)+\frac{\partial S}{\partial x}(x_1, x_2)=0, \dots , \frac{\partial S}{\partial y}(x_{m-2}, x_{m-1})+\frac{\partial S}{\partial x}(x_{m-1}, x_m)=0.$$
 This implies that $(x_0, \dots, x_m)$ is the projection of a unique orbit, that is: 
 $$(x_0, -\frac{\partial S}{\partial x}(x_0, x_1)), (x_1, -\frac{\partial S}{\partial x}(x_1, x_2)= \frac{\partial S}{\partial y}(x_0, x_1)), \dots, (x_m,  \frac{\partial S}{\partial y}(x_{m-1}, x_m)).$$
Moreover,
 $$\begin{matrix}\Ac_m(x',y')&\leq   \bar S(x', x_1)+\dots  +\bar S(x_{m-1}, y')\hfill\\
 &\leq\Ac_m(x, y)+(\bar S(x', x_1)-\bar S(x,x_1))+(\bar S(x_{m-1}, y')-\bar S(x_{m-1}, x))\hfill\\
 &\leq \Ac_m(x, y)+\frac{\partial S}{\partial x}(x, x_1)(x'-x)+\frac{\partial S}{\partial y}(x_{m-1}, y)(y'-y)\\
& \hfill +K(\| x-x'\|^2+\| y-y'\|^2)
 \end{matrix}
 $$
 hence  $( \frac{\partial S}{\partial x}(x_0, x_1), \frac{\partial S}{\partial y}(x_{m-1}, x_m))$ is a superdifferential for $\Ac_m$ at $(x,y)$.\\

 Let us now assume that $\Sigma_m(x,y)$ contains at least two distinct elements $(x_0, \dots, x_m)$ and $(y_0, \dots , y_m)$. We know that they are the projections of two distinct orbits, one joining $(x, -\frac{\partial S}{\partial x}(x_0, x_1))$ to $(y,  \frac{\partial S}{\partial y}(x_{m-1}, x_m))$ and the other one joining $(x, -\frac{\partial S}{\partial x}(y_0, y_1))$ to $(y,  \frac{\partial S}{\partial y}(y_{m-1}, y_m))$. Because the orbits are distinct, the points are not the same and then $\frac{\partial S}{\partial x}(x_0, x_1)\not= \frac{\partial S}{\partial x}(y_0, y_1)$ and $\frac{\partial S}{\partial y}(x_{m-1}, x_m)\not= \frac{\partial S}{\partial y}(y_{m-1}, y_m)$. Hence $\Ac_m$ has two distinct superderivatives with respect to $x$ and two distinct superderivatives with respect to $y$ at $(x,y)$.\\
 
 Let us   assume that $\Sigma_m(x,y)$ contains exactly one  element  $(x_0, \dots, x_m)$. Let $(x', y')$ be close to $(x,y)$. Then every element  $(x'_0, \dots, x'_m)$ of $\Sigma_m(x',y')$ is close to $(x_0, \dots, x_m)$ and we have:
 $$\begin{matrix} \Ac_m(x', y')&=\displaystyle{\sum_{i=1}^m}\bar S (x'_{i-1}, x'_i)  =\bar S(x, x'_1)+\displaystyle{\sum_{i=2}^{m-1}}\bar S (x'_{i-1}, x'_i)+\hfill\\ & +\bar S(x'_{m-1}, y)+ \bar S(x', x'_1)-\bar S(x, x'_1)+\bar S(x'_{m-1}, y')-\bar S(x'_{m-1}, y)\hfill\\
 &\geq \Ac_m(x, y)+\bar S(x', x'_1)-\bar S(x , x'_1)+\bar S(x'_{m-1}, y')-\bar S(x'_{m-1}, y)\hfill\\
 &\geq \Ac_m(x, y)+\frac{\partial\bar S}{\partial x}(x, x'_1)(x'-x)+\frac{\partial\bar S}{\partial y}(x'_{m-1}, y)(y'-y)+o(\| x-x'\|)+o(\| y'-y\|)\hfill\\
 &\geq \Ac_m(x, y)+\frac{\partial\bar S}{\partial x}(x, x_1)(x'-x)+\frac{\partial\bar S}{\partial y}(x_{m-1}, y)(y'-y)+o(\| x-x'\|)+o(\| y'-y\|).\hfill\end{matrix}.$$
  This proves that $\Ac_m$ is differentiable at $(x,y)$. 
\end{proof}

\begin{nota}
At every $x\in\R^n$ we denote by $\Vc(x)$ the fiber $\{ x\}\times \R^n $ of $\R^n\times \R^n$.
\end{nota}

\begin{prop}\label{Pvertical}
Let $(x,y)$ be a point of differentiability of $\Ac_m$. Then $(y, \frac{\partial \Ac_m}{\partial y}(x,y))\in F^m(\Vc(x))$ and $(x, -\frac{\partial \Ac_m}{\partial x}(x,y))\in F^{-m}(\Vc(y))$.

\end{prop}

\begin{demo} We use proposition \ref{Ppotential}.  As $(x,y)$ is a point of differentiability of $\Ac_m$, $\Sigma_m(x,y)=\{ (x_0, \dots , x_m)\}$ has only one element and this is the projection of the $F$-orbit
$$(x_0, -\frac{\partial S}{\partial x}(x_0, x_1)), (x_1, -\frac{\partial S}{\partial x}(x_1, x_2)= \frac{\partial S}{\partial y}(x_0, x_1)), \dots, (x_m,  \frac{\partial S}{\partial y}(x_{m-1}, x_m)).$$
Moerover, we have $\frac{\partial \Ac_m}{\partial x}(x,y)=\frac{\partial S}{\partial x}(x, x_1)$ and $\frac{\partial \Ac_m}{\partial y}(x,y)=\frac{\partial S}{\partial y}(x_{m-1}, y)$. We deduce that $F^m(x, -\frac{\partial \Ac_m}{\partial x}(x,y))=(y, \frac{\partial \Ac_m}{\partial y}(x,y))$ and then  proposition \ref{Pvertical}.
\end{demo}
\begin{cor}\label{C1}
We assume that a piece of  orbit $(x_i,y_i)_{i\in [0, m+1]}$ for $F$  is minimizing. Then 
 $\Ac_m$ is as regular as $F$ is in a neighbourhood of $(x_0,x_m)$ and  in a neighborhood of $(x_1, x_{m+1})$.
\end{cor}
\begin{demo}
We prove the first assertion. \\
Let us prove that  $DF^{m}(V(x_0))$ is transverse to $V(x_m)$.  We use the results that are contained in section 2.3. of \cite{Arn4} (especially proposition 6). Let us use the notation:
$$\bar \Fc(y_0, \dots , y_{m+1})=\sum_{i=0}^m\bar S(y_i, y_{i+1}) $$
and $(x_0, \dots , x_{m+1})$ is a minimizer of $\bar\Fc$ among the $(y_0, \dots , y_{m+1})$ such that $y_0=x_0$ and $y_{m+1}=x_{m+1}$.  
We denote by $\Hc=\Hc(x_0, \dots, x_{m+1})$ the Hessian of $\Fc$ with fixed ends at $(x_0, \dots, x_{m+1})$. Then it is positive semidefinite.  The kernel of $\Hc$ is the set of projections $(\delta x_i)_{1\leq i\leq m}$ of infinitesimal orbits $(\delta x_i, \delta y_i)_{1\leq i\leq m}$ along the orbit $(x_i,y_i)_{1\leq i\leq m}$ such that their extension $(\delta x_i, \delta y_i)_{i\in\Z}$ satisfies $\delta x_0=0$ and $\delta x_{m+1}=0$.\\
Let us assume that  $DF^{m}(V(x_0))$ is not transverse to $V(x_m)$. Then there exists an infinitesimal orbit $(\delta x_i, \delta y_i)_{0\leq i\leq m}$ that is not the $(0, 0)$ orbit and that satisfies $\delta x_0=\delta x_m=0$.  Then $( 0, \delta x_1, \dots , \delta x_{m-1}, 0,0)$ is in the isotropic cone for $\Hc(x_{0}, \dots, x_{m+1})$ and because $\Hc(x_0, \dots, x_{m+1})$ is positive semi-definite, $( 0, \delta x_1, \dots , \delta x_{m-1}, 0,0)$ is in the kernel of $( 0, \delta x_1, \dots , \delta x_{m-1}, 0,0)$. This implies that it is an infinitesimal orbit and then the $0$-orbit.\\

We have then proved that $DF^{m}(V(x_0))$ is transverse to $V(x_m)$ and this implies that $DF^{-m}(V(x_m))$ is transverse to $V(x_0)$. Hence $F^m(\Vc(x_0))$ (resp. $F^{-m}(\Vc(x_m))$) is a manifold that is a graph as smooth as $F$ is in a neighborhood of $(x_m, y_m)$ (resp. $(x_0, y_0)$).\\
If $(x'_0, x_m')$ is closed to $(x_0, x_m)$, we have noticed that every element of $\Sigma(x'_0, x_m')$ is closed to the unique element of $\Sigma (x_0, x_m)$. Hence it corresponds to an orbit $(x'_i, y'_i)_{0\leq i\leq m}$ that is closed to $(x_i, y_i)_{0\leq i\leq m}$. Moreover, $F^m(\Vc(x'_0))$ (resp. $F^{-m}(\Vc(x'_m))$) is a manifold that is a graph as smooth as $F$ is in a neighborhood of $(x'_m, y'_m)$ (resp. $(x'_0, y'_0)$) because it is close to $F^m(\Vc(x_0))$  (resp. $F^{-m}(\Vc(x_m))$). Hence there is only one choice for $(x_0', y_0')$ above $x_0'$ close to $x_0$   on $F^{-m}(\Vc(x'_m))$ and it smoothly depends on $(x'_0, x'_m)$ and we have the same result for the choice of $y'_m$. This means that $\Sigma_m(x'_0, x'_m)$ has only one element, hence $\Ac_m$ is differentiable at $(x_0', x_m')$. Morever,   $\frac{\partial \Ac_m}{\partial x}(x'_0, x'_m)=-y'_0$ and $\frac{\partial \Ac_m}{\partial y}(x'0, x'_m)=y'_m$ smoothly depend on $(x'_0, x'_m)$.

 \end{demo}

\subsection{Comparison between Ma\~ n\'e's potential and subactions}

  A consequence of the definition of a subaction is that if $u:\R^n\rightarrow \R$ is a subaction, then: $\forall x, y\in\R^n, u(y)-u(x)\leq \Ac_m(x, y)$.
  \begin{prop}
  Let $u_-:\R^n\rightarrow \R$ be a backward calibrated subaction and let $u_+:\R^n\rightarrow \R$ be a forward calibrated subaction. Let $x_0$ be a point of differentiability for $u_-$ (resp. $u_+$).   Then the backward (resp. forward) orbit of $(x_0, du_-(x_0))$ (resp. $(x_0, du_+(x_0))$) is on the graph of $du_-$ (resp. $du_+$) and is denoted by $(x_i, du_-(x_i))_{i\in\N}$ (resp.  $(x_i, du_+(x_i))_{i\in\N}$). Then $(x_i)$ is strongly minimizing, $\Ac_m$ is differentiable at every $(x_m, x_0)$ (resp. $(x_0, x_m)$) with $m\geq 1$ and for every $x\in\R^n$
  $$u_-(x)-u_-(x_0)-du_-(x_0)(x-x_0)\leq \Ac_m(x_m, x)-\Ac_m(x_m, x_0)-\frac{\partial \Ac_m}{\partial y}(x_m, x_0)(x-x_0)$$
(resp.  
  $$ u_+(x)-u_+(x_0)-du_+(x_0)(x-x_0)\geq \Ac_m(x_0, x_m)- \Ac_m(x, x_m)+\frac{\partial \Ac_m}{\partial x}(x_0, x_m)(x-x_0)).$$
  \end{prop}
  
\begin{demo}
  We prove the result for $u_-$. We assume that $x_0$ is a point of differentiability for $u_-$.\\
   We deduce from the remark after proposition \ref{Pbackward} that the backward   orbit of $(x_0, du_-(x_0))$ is on the graph of $du_-$  and we denote it by $(x_i, du_-(x_i))_{i\in\N}$ . We deduce from the same remark that $(x_i)$ is strongly minimizing. We  deduce from proposition \ref{Pbackward} that $\Sigma (x_i)$ has only one element. Hence $\Sigma (x_m, x_0)$ has only one element and then $\Ac_m$ is differentiable at $(x_m, x_0)$. \\
 As $u_-$ is a subaction, we have  
 $$\forall x\in\R^n, u_-(x)-u_-(x_0)=u_-(x)-u_-(x_m)+u_-(x_m)-u_-(x_0)\leq \Ac_m(x_m, x)-\Ac_m(x_m, x_0)$$
 because $u_-(x_m)-u_-(x_0)=\Ac_m(x_m, x_0)$. As the two functions vanish for $x=x_0$ and are differentiable with respect to $x$, we deduce $u_-'(x_0)=\frac{\partial \Ac_m}{\partial y}(x_m, x_0)$ and then the wanted inequality.
  \end{demo} 
  \begin{prop} \label{Palg}
We assume that $(u_-, u_+)$ is a pair of conjugate calibrated subaction.  Let $x\in\Ic(u_-, u_+)$ be a point, $(y_n)$ be a sequence of points of $\R^n$ converging to $x$, and $(\lambda_n)$ be a sequence of positive real numbers so that the two limits (written in charts) $\displaystyle{\lim_{n\rightarrow \infty} \frac{y_n-x}{\lambda_n}=X}$ and $\displaystyle{Y=\lim_{n\rightarrow \infty} \frac{du_-(y_n)-du_-(x)}{\lambda_n} }$ (resp. $\displaystyle{ \lim_{n\rightarrow \infty} \frac{du_+(y_n)-du_+(x)}{\lambda_n} }$)  exist. Then we have:
$\forall k\in\R^n,$
$$ Y.k \leq \frac{1}{2} \bigg(    \frac{\partial^2 \Ac_m}{\partial y^2}(x_{-m}, x)(k,k)+ \frac{\partial^2 \Ac_m}{\partial y^2}(x_{-m}, x)(X,X)+ \frac{\partial^2 \Ac_m}{\partial x^2}(x, x_m) (X-k,X-k)\bigg)$$
where $(x_i, du_-(x_i))_{i\in\Z}$ is the  orbit of $(x, du_-(x))$
(resp:
$\forall k\in\R^n,$
$$  \frac{1}{2}\bigg(  -\frac{\partial^2 \Ac_m}{\partial x^2}(x, x_m)(k,k)-\frac{\partial^2 \Ac_m}{\partial x^2}(x, x_m)(X,X)- \frac{\partial^2 \Ac_m}{\partial y^2}(x_{-m}, x)(k-X,k-X) {\bigg)}\leq Y.k  )$$
\end{prop}
\begin{demo} The proof is an adapted version of the proof of proposition 18 in \cite{Arn2}. We just prove the first inequality. \\
Let $x\in\Ic(u_-, u_+)$ and let $z$ be a point of differentiability of $u_-$. We denote the negative orbit of $(z, du_-(z))$ by $(z_{-i}, du_-(z_{-i})_{i\in\N}$.Then we have:
\begin{enumerate}
\item[$\bullet$] $u_-({z}+h)-u_-(z)-du_-({z})h\leq   \Ac_m(z_{-m}, z+h)-\Ac_m(z_{-m}, z)-\frac{\partial \Ac_m}{\partial y}(z_{-m}, z)h$;
\item[$\bullet$] $u_-({z})-u_-(x)-du_-(x)({z}-x)\leq  \Ac_m(x_{-m} ,z)-\Ac_m(x_{-m}, x)-\frac{\partial \Ac_m}{\partial y}(x_{-m}, x)(z-x))$;
\item[$\bullet$] $
 \Ac_m(x, x_m)- \Ac_m(z+h, x_m)+\frac{\partial \Ac_m}{\partial x}(x, x_m)(z+h-x)\leq u_+({z}+h)-u_+(x)-du_+(x)({z}+h-x)$.
\end{enumerate}
Hence, by adding these three inequalities and  using that $u_-(x)=u_+(x)$, $du_-(x)=du_+(x)$  and $u_+\leq u_-$:\\

\noindent $(du_-(x) -du_-({z}))h$\\
\hglue 1truecm$ \leq  \Ac_m(z_{-m}, z+h)-\Ac_m(z_{-m}, z)-\frac{\partial \Ac_m}{\partial y}(z_{-m}, z)h   + \Ac_m(x_{-m} ,z)-\Ac_m(x_{-m}, x)$\\
\hglue 2truecm$-\frac{\partial \Ac_m}{\partial y}(x_{-m}, x)(z-x))    - \Ac_m(x, x_m)+ \Ac_m(z+h, x_m)-\frac{\partial \Ac_m}{\partial x}(x, x_m)(z+h-x)$.\\

We now consider a sequence $(y_k)$ of points of differentiability of $u_-$ that converges to $x$ such that $\forall k, y_k \not=x$,  a vector $K$ with fixed norm $\| K\|=\mu>0$ and the sequence $(h_k)=(\lambda_k K)$ where $(\lambda_k)$ is a sequence of positive numbers tending to zero. we denote by $(z_{-i}^k, du_-(z_{-i}^k))$ the backward orbit of $(y_k, du_-(y_k))$ for $F$. We have proved that:

\noindent $(du_-(x) -du_-({y_k}))h$\\
\hglue 1truecm$ \leq  \Ac_m(z^k_{-m}, y_k+h_k)-\Ac_m(z^k_{-m}, y_k)-\frac{\partial \Ac_m}{\partial y}(z^k_{-m},y_k)h_k   + \Ac_m(x_{-m} ,y_k)-\Ac_m(x_{-m}, x)$\\
\hglue 1truecm$-\frac{\partial \Ac_m}{\partial y}(x_{-m}, x)(y_k-x))    - \Ac_m(x, x_m)+ \Ac_m(y_k+h_k, x_m)-\frac{\partial \Ac_m}{\partial x}(x, x_m)(y_k+h_k-x)$.\\
We assume that $\displaystyle{\lim_{k\rightarrow +\infty} \frac{y_k-x}{\lambda_k}=X}$ and $\displaystyle{\lim_{k\rightarrow +\infty} \frac{du_-(y_k)-du_-(x)}{\lambda_k}=Y}$. We have proved in corollary \ref{C1} that $\Ac_m$ is as regular as $F$ is in a neighborhood of $(x_{-m}, x)$, $(z^k_{-m}, y_k)$ and $(x, x_m)$.  Moreover, we have the following lemma that is lemma 18 in \cite{Arn2}:
\begin{lemma}\label{Aubrylipschitz}There exists a constant $K>0$  such that, for every $q\in\Ic(u_-, u_+)$ and every ${q'}\in M$ where $u_-$ (resp. $u_+$) is differentiable, then $\| du_-(q)-du_-({q'})\|\leq K\| q-{q'}\|$ (resp. $\| du_+(q)-du_+({q'})\|\leq K\| q-{q'}\|$ ). In particular, $du_-$ and $du_+$ are continuous at every point of $\Ic(u_-, u_+)$.
\end{lemma}
This lemma implies that $(y_k, du_-(y_k))$ is closed to $(x, du_-(x))$ and then that $z_{-m}^k$ is close to $x_{-m}$. Hence we obtain: 

 $$\begin{matrix} (du_-(x) -&du_-({y_k}))h_k \leq \frac{1}{2} ( \frac{\partial^2\Ac_m}{\partial y^2} (z_{-m}^k, y_k)(h_k, h_k) +      \frac{\partial^2\Ac_m}{\partial y^2} (x_{-m}, x)(y_k-x, y_k-x)\hfill\\ 
 &+\frac{\partial^2\Ac_m}{\partial x^2}(x, x_m) (y_k+h_k-x, y_k+h_k-x) +o(\| h_k\|^2+\| y_k+h_k-x\|^2) ) .\end{matrix}$$
We multiply by $\frac{1}{\lambda_n^2}$ and take the limit and obtain
$$-Y.K\leq \frac{1}{2}( \frac{\partial^2\Ac_m}{\partial y^2}(x_{-m}, x )(K, K)+\frac{\partial^2\Ac_m}{\partial y^2}(x_{-m}, x)(X, X)+\frac{\partial^2\Ac_m}{\partial x^2}(x, x_m)(X+K, X+K)).$$
Changing $K$ into $-K$, we obtain the wanted inequality.

\end{demo}

 \subsection{Links between the tangent cone to the support of a strongly minimizing measure  and the Green bundles}\label{sscone}
 The notion of contingent cone was introduced by G.~Bouligand in \cite{Bou}. 
 
\begin{defin}
Let $A\subset \R^n\times \R^n$ be a subset of $\R^n\times \R^n$ and let $a\in A$ be  a point of $A$. Then the {\em contingent cone } to $A$ at $a$ is defined as being the set of all the limit points of the sequences $t_k(a_k-a)$ where $(t_k)$ is a sequence of real numbers and $(a_k)$ is a sequence of elements of $A$ that converges to $a$. This cone is denoted by $C_aA$ and it is a subset of $T_a(\R^n\times \R^n)$.

\end{defin}
We introduce an extension to this definition that is

\begin{defin}
Let $A\subset \R^n\times \R^n$ be a subset of $\R^n\times \R^n$ and let $a\in A$ be  a point of $A$.  Then the {\em limit contingent cone} to $A$ at $a$ is the set of the limit points of sequences $v_k\in C_{a_k}A$ where $(a_k)$ is any sequence of points of $A$ that converges to $a$. It is denoted by $\widetilde C_aA$. 

\end{defin}
In general , these tangent cones are not Lagrangian subspaces. Because we need to compare them to Lagrangian subspaces, we give a definition:

\begin{defin}
Let $\Lc_-\leq \Lc_+$ be two Lagrangian subspaces of $T_x( \R^n\times \R^n) $ that are transverse to the vertical. If $v\in T_x (\R^n\times \R^n )$ is a vector, we say that $v$ is between $\Lc_-$ and $\Lc_+$ and write $\Lc_-\leq v\leq \Lc_+$ if there exists a third Lagrangian subspace in $T_x( \R^n\times \R^n) $ such that:
\begin{enumerate}
\item[$\bullet$] $v\in\Lc$;
\item[$\bullet$] $\Lc_-\leq \Lc\leq \Lc_+$.
\end{enumerate}
A part $B$ of $T_x( \R^n\times \R^n) $ is between $\Lc_-$ and $\Lc_+$ if $\forall v\in  B, \Lc_-\leq v\leq \Lc_+$. Then we write $\Lc_-\leq B\leq \Lc_+$.
\end{defin} 
We introduce   the two modified Green bundles. We use the constant $c_0=\frac{\sqrt{13}}{3}-\frac{5}{6}$.

\begin{defin}
We denote by $S_\pm(x):\R_n\rightarrow \R^n$ the linear operator such that $G_\pm(x)$ is the graph of $S_\pm(x)$: $G_\pm(x)=\{ (v, S_\pm(x)v); v\in\R^n\}$. Then the {\em modified Green bundles} $G_\pm$ are defined by:
$$\widetilde G_-(x)=\{ (v, (S_-(x)-c_0(S_+(x)-S_-(x)))v); v\in \R^n\}$${\rm and}$$ \widetilde G_+(x)=\{ (v, (S_+(x)+c_0(S_+(x)-S_-(x)))v); v\in \R^n\}.$$
\end{defin}

\begin{prop}
Let $(u_-, u_+)$ be a pair of conjugate calibrated subactions. Then 
$$\forall x\in \Ic(u_-, u_+), \widetilde G_-(x, du_-(x))\leq  \widetilde C_{(x, du_-(x))}\widetilde \Ic(u_-, u_+)\leq \widetilde G_+(x, du_-(x)).$$
\end{prop}
\begin{demo}
A consequence of proposition \ref{Palg} and proposition \ref{Pbilin} is that: 
$$\forall x\in \Ic(u_-, u_+), \widetilde G_-(x, du_-(x))\leq   C_{(x, du_-(x))}\widetilde \Ic(u_-, u_+)\leq \widetilde G_+(x, du_-(x)).$$
Then the conclusion of the proposition comes from the definition of the limit contingent cone and the semicontinuity property of the Green bundles (see for example \cite{Arn5}) and then  of the modified Green bundles.
\end{demo}
As $\Mc(S)\subset \widetilde \Ic(u_-, u_+)$, we deduce the following corollary and then theorem \ref{Tcone}.
\begin{cor}
We have: $\forall x\in \Mc(S), \widetilde G_-(x)\leq   \widetilde C_{x}\Mc(S)\leq \widetilde G_+(x, du_-(x)).$
\end{cor}
\begin{defin}
A subset $A$ of $\R^n\times \R^n$ is $C^1$-isotropic at some point $a\in A$ if $\widetilde C_aA$ is contained in some Lagrangian subspace.
\end{defin}
For example, a $C^1$ submanifold is $C^1$-isotropic if it isotropic. \\

Corollary \ref{Cisotropic} that is given in the introduction is just a consequence of theorem \ref{Tcone} and theorem \ref{nbexp}.

 \section{Appendix}
 \subsection{Comparison of Lagrangian subspaces}\label{sscomp}
 Let us assume that $(E, \omega)$ is a symplectic $2n$-dimensional space. Let $L_1$, $L_2$ be two transverse Lagrangian subspaces of $E$. Then the set the  Lagrangian subspaces of $E$ that are transverse to $L_1$ and $L_2$ is open in the Grassmann space $\Lc$ of the Lagrangian subspaces of $E$. Moreover it has exactly $n+1$ connected component. Let us be more precise.  
 
 \begin{nota}
 If $L\in\Lc$ is transverse to $L_2$, then it is the graph of a linear map $\ell:L_1\rightarrow L_2$. We then define a quadratic form $q(L_1, L_2, ;L)$ on $L_1$ by:
 $$\forall v\in L_1, q(L_1, L_2; L)(v)=\omega (v, \ell(v)).$$
  Then $L$ is transverse to both $L_1$ and $L_2$ if and only if $q(L_1, L_2; L)$ is non-degenerate and the connected components of the  set the  Lagrangian subspaces of $E$ that are transverse to $L_1$ and $L_2$ correspond to the signature of this quadratic form.\\
  We will denote by $\Pc(L_1, L_2)$ the set of the $L\in\Lc$ that correspond to a positive definite quadratic form.
 \end{nota}

 \begin{prop}\label{propapp} Let $L_1, L_2\in\Lc$ be two transverse Lagrangian subspaces of $E$. Then
 \begin{enumerate}
 \item if $M:E\rightarrow E$ is a symplectic isomorphism, we have: $M(\Pc(L_1, L_2))=\Pc(M(L_1), M(L_2))$;
 \item if $L\in\Pc(L_1, L_2)$, then $\Pc(L_1, L)\cup \Pc(L, L_2)\subset \Pc(L_1, L_2)$.
 \end{enumerate}
 \end{prop}
 \demo
 The proof of the first assertion  is elementary.\\
 For the second one, let us begin by proving that  $\Pc(L_1, L)\subset \Pc(L_1, L_2)$. Let $W\in \Pc(L_1, L)$. For $w\in W\backslash\{ 0\}$,  we write $w=\ell_1+\ell$ with $\ell_1\in L_1$, $\ell\in L$.  Then we have: $\omega(\ell_1, \ell)>0$.  As $\ell\in L\backslash \{ 0\} $ and $L\in \Pc(L_1, L_2)$, we can write $\ell =\ell_1'+\ell_2'$ with $\ell_i'\in L_i$ and we have $\omega (\ell'_1, \ell'_2)>0$. \\
 Finally we have proved that $w=(\ell_1+\ell_1')+\ell_2'$ with $\ell_1+\ell_1'\in L_1$ and $\ell_2'\in L_2$ and $\omega(\ell_1+\ell_1', \ell_2')=\omega(\ell_1, \ell_2')+\omega(\ell_1', \ell_2')= \omega(\ell_1, \ell_1'+\ell_2')+\omega(\ell_1', \ell_2')>0$.\\
 the proof of the second inclusion is very similar.
 \enddemo
 
 In the particular case where $E=T_x\A_n=\R^n\times\R^n$, we define an order relation on the set $\Hc$ of Lagrangian subspaces that are transverse to $V(x)$ in the following way.
 
 \begin{defin}
 If $L_1, L_2\in\Hc$, 
 \begin{enumerate}
 \item we say that $L_1$ is stricly under $L_2$ and write $L_1<L_2$ if $L_2\in \Pc(L_1, V(x))$; 
 \item  we say that $L_1$ is  under $L_2$ and write $L_1\leq L_2$ if $L_2$ is in the closure of  $\Pc(L_1, V(x))$.
 \end{enumerate}
 \end{defin}
Note that $L_1\leq L_2$ if and only if $q(L_1, V(x); L_2)$ is positive semi-definite.
A consequence of proposition \ref{propapp} is that $<$ and $\leq$ are transitive.\\
We can then define what is a decreasing or increasing sequence of elements of $\Hc$.

\begin{prop}\label{proappbis}
If $L_1, L_2, L_3\in\Hc$, if $L_1<L_2$ and $L_3\in\Pc(L_1, L_2)$, then $L_1<L_3$ and $L_3<L_2$. 
\end{prop}
\demo
Let us prove the first inequality.  We assume that $L_1<L_2$, i.e.  $L_2\in \Pc(L_1, V(x))$.  We know by proposition \ref{propapp} that   $\Pc (L_1, L_2)\cup \Pc(L_2, V(x))\subset \Pc (L_1, V(x))$. We deduce that $L_3\in \Pc(L_1, V(x))$ i.e. $L_1<L_3$.\\
We  explain how to prove the second inequality. We choose a basis $(e_1,\dots,e_n)$ of $L_3$ and complete it with $f_1, \dots, f_n\in V(x)$ in such a way that the basis is symplectic.\\
Then there exist two symmetric matrices $S_1$ and $S_2$ such that $L_i$ is the graph of the linear map $\phi_i: L_3\rightarrow V(x)$ with matrix $S_i$ in the bases $(e_1, \dots, e_n)$, $(f_1, \dots, f_n)$. Because $L_1<L_3$, we know that $S_1$ is negative definite. We want to prove that $S_2$ is positive definite.\\
Let us write that $L_3\in\Pc(L_1, L_2)$.  This means that for all $v\in \R^n\backslash\{ 0\}$, if $(v, 0)=(v_1, S_1v_1)+(v_2, S_2v_2)$, then ${}^tv_1S_2v_2-{}^tv_2S_1v_1>0$. This can be reformulated in the following way.
$$\forall w\in\R^n, -{}^twS_2S_1^{-1}S_2w+{}^twS_2 w>0.$$
Let $s$ be the positive definite matrix such that $s^2=-S_1$. If $s_2=s^{-1}S_2s^{-1}$, we obtain
$$\forall u\in \R^n, {}^tus_2^2u+{}^tus_2u>0.$$
If $\lambda_1, \dots, \lambda_n$ are the eigenvalues of $s_2$, we deduce that $\lambda_i^2+\lambda_i>0$ i.e. $\lambda_i<-1$ or $\lambda_i>0$. Moreover, we know that $L_1<L_2$, hence $0<-S_1+S_2$, i.e. $0<{\bf 1}_n+s_2$ and $\lambda_i>-1$. We deduce that $\lambda_i>0$ and $S_2$ is positive definite.
\enddemo

\begin{prop}\label{proappter}
If $L_1, L_2, L_3\in\Hc$, if $L_1<L_3<L_2$ then $L_3\in\Pc(L_1, L_2)$. 
\end{prop}
\demo As in the proof of  proposition \ref{proappbis}, we choose a basis $(e_1,\dots,e_n)$ of $L_3$ and complete it with $f_1, \dots, f_n\in V(x)$ in such a way that the basis is symplectic.  Then there exist two symmetric matrices $S_1$ and $S_2$ such that $L_i$ is the graph of the linear map $\phi_i: L_3\rightarrow V(x)$ with matrix $S_i$ in the bases $(e_1, \dots, e_n)$, $(f_1, \dots, f_n)$. We know that $S_1$ is negative definite and $S_2$ is positive definite.\\
We want to prove  that $L_3\in\Pc(L_1, L_2)$.  This means that for all $v\in \R^n\backslash\{ 0\}$, if $(v, 0)=(v_1, S_1v_1)+(v_2, S_2v_2)$, then ${}^tv_1S_2v_2-{}^tv_2S_1v_1>0$.  As $S_2$ is positive definite and $S_1$ is negative definite, the conclusion is straightforward.
\enddemo

\subsection{A result  in bilinear algebra}

\begin{prop}\label{Pbilin}
Let $Q_-$, $Q_+$ be two quadratic forms on $\R^n$ such that $Q_-\leq Q_+$ and let $(X, Y)\in \R^n\times \R^n$ be such that:
$$\forall K\in\R^n, Y.K \leq \frac{1}{2} \bigg(   Q_+(K,K)+ Q_+(X,X)- Q_-(X-K,X-K)\bigg)$$
and 
$$\forall K\in\R^n,  \frac{1}{2}\bigg(  Q_-(K,K)+Q_-(X,X)- Q_+(K-X,K-X) {\bigg)}\leq Y.K .$$
Then there exists a quadratic form $\sigma$ such that:
\begin{enumerate}
\item[$\bullet$] $Q_--(\frac{\sqrt{13}}{3}-\frac{5}{6})(Q_+-Q_-)\leq \sigma \leq Q_++(\frac{\sqrt{13}}{3}-\frac{5}{6})(Q_+-Q_-)$;
\item[$\bullet$] $Y={}^t\sigma(X, .)$.
\end{enumerate}
\end{prop}

\begin{remk}
1) Note that $\frac{\sqrt{13}}{3}-\frac{5}{6}<\frac{1}{2}$, hence we obtain the same inequalities by replacing $\frac{\sqrt{13}}{3}-\frac{5}{6}$ by $\frac{1}{2}$.\\
2) We gave in \cite{Arn2} an example in dimension $n=2$ that proves that in general, we cannot improve the first point into $Q_-\leq \sigma\leq Q_+$.  
\end{remk}

\begin{nota}
$\Delta Q= Q_+-Q_-$; $\Delta Y_+=Y-{}^tQ_+(X, .)$ and $\Delta Y_-=Y-{}^tQ_-(X, .)$. \\
We use the constant: $c_0=\frac{\sqrt{13}}{3}-\frac{5}{6}$.
\end{nota}
\noindent Note that  $\Delta Y_--\Delta Y_+={}^t\Delta Q(X, .)$. 

\begin{demo}
Using the above notations, we rewrite the two  inequalities:
$$\forall K\in\R^n, \Delta Y_+.K\leq \frac{1}{2}\Delta Q(X-K, X-K)\quad{\rm and}\quad \Delta Y_-.K\geq -\frac{1}{2}\Delta Q(X-K, X-K).
$$
We deduce that $\Delta Y_+, \Delta Y_-\in \Im^t\Delta Q=(\ker \Delta Q)^\bot$. We then use the restriction of $\Delta Q$ to $\Im^t\Delta Q=\R^d$, hence $\Delta Q$ is positive definite and we want to prove that there exists a quadratic form $\sigma$ on $\R^d$ such that $-(1+c_0)\Delta Q \leq \sigma\leq c_0 \Delta Q$ and $\Delta Y_+ ={}^t \sigma (X,  )$.

As $\Delta Q$ is positive definite, there exists a symmetric automorphism $L: \R^d\rightarrow\R^d$ such that $\Delta Q(L(X))=\| X\|^2$ ($\| .\|$ is the usual Euclidean norm). We introduce the notations $x=L^{-1}X$,  $y_+={}^tL\Delta Y_+$ and $y_-={}^tL\Delta Y_-$. Note that $y_--y_+=x$. The inequalities are rewritten as:
$$\forall k\in \R^d, y_+.k\leq\frac{1}{2}\| k-x\|^2\quad{\rm and}\quad y_-.k\geq -\| x-k\|^2.$$
We now want to find $\eta=\sigma\circ L$ such that $y_+={}^t\eta (x,.)$ and  $-(1+c_0)\|.\|^2 \leq \eta\leq c_0 \|.\|^2.$ Using an orthogonal change of basis, we can assume that $x=(\mu, 0, \dots, 0)$ and we can multiply all the inequalities by $\mu^2$ and assume that $\mu=1$.
We use the notations $y_+=(y_i)_{1\leq i\leq d}$,  and $k=(k_i)_{1\leq i\leq d}$. We have $x=(1, 0, \dots, 0)$. Then the inequalities become:
$$\sum_{i=1}^d y_i.k_i\leq \frac{1}{2}(k_1-1)^2+\frac{1}{2}\sum_{i=2}^dk_i^2\quad{\rm and}\quad k_1+\sum_{i=1}^d y_i.k_i\geq -\frac{1}{2}(k_1-1)^2-\frac{1}{2}\sum_{i=2}^dk_i^2.$$
They can be rewritten as follows
 $$(k_1-1-y_1)^2 +1+\sum_{i=2}^d(k_i-y_1)^2\geq  (y_1+1)^2+\sum_{i=2}^dy_i^2$$
 and
 $$\sum_{i=1}^d(k_i+y_i)^2+1\geq \sum_{i=1}^d y_i^2.
 $$
 As $(k_i)_{1\leq i\leq d}$ can be any element of $\R^d$, this is equivalent to:
 $$ (y_1+1)^2+\sum_{i=2}^dy_i^2\leq 1\quad{\rm and}\quad 1\geq \sum_{i=1}^d y_i^2.
 $$
Then  we choose the quadratic form $\eta$. Its matrix in the canonical basis is 
$$S=\begin{pmatrix}y_1&y_2&y_3&\dots &y_{d-1}&y_d\\
y_2&-\frac{1}{2}&0&\dots &0&0\\
.&.&.&\dots&.&.\\
y_d&0&0&\dots&0&-\frac{1}{2}
\end{pmatrix}
$$
i.e. the only entries that may be non-zero are on the first line, on the first column and on the diagonal. If ${\bf 1}$ is the identity matrix, we have to prove that $c_0{\bf 1}-S$ and   $(1+c_0){\bf 1}+S$ are    positive semidefinite. We have 
$$c_0{\bf 1}-S= \begin{pmatrix}c_0-y_1&-y_2&-y_3&\dots &-y_{d-1}&-y_d\\
-y_2&c_0+\frac{1}{2}&0&\dots &0&0\\
.&.&.&\dots&.&.\\
-y_d&0&0&\dots&0&c_0+\frac{1}{2}
\end{pmatrix}$$
The restriction of $c_0\| .\|^2-\eta$ to $\{0\}\times \R^{d-1}$ is positive definite. Hence to prove that this quadratic form is positive, we only have to prove that the determinant of $c_0{\bf 1}-S$ is non-negative. We then compute it. Note that when $d=1$, we have: $\delta(1)=c_0-y_1$. Moreover, if $d\geq 2$, we have
$$
\delta(d)=\det(c_0{\bf 1}-S)=(c_0+\frac{1}{2})\delta(d-1) + (-1)^{d} y_d\det \begin{pmatrix}
-y_2&c_0+\frac{1}{2}&0&\dots &0\\
.&.&.&\dots&c_0+\frac{1}{2}\\
-y_d&0&0&\dots&0
\end{pmatrix}
$$
and thus $$\delta (d)=(c_0+\frac{1}{2})\delta(d-1) + (-1)^{d}(-1)^{d-1} y_d^2(c_0+\frac{1}{2})^{d-2}=(c_0+\frac{1}{2})\delta(d-1) - y_d^2(c_0+\frac{1}{2})^{d-2} .$$ We finally deduce:
$$\delta (d)=(c_0+\frac{1}{2})^{d-1}\left((c_0+\frac{1}{2})(c_0-y_1)-\sum_{i=2}^dy_i^2\right).$$
We have proved that $\displaystyle{(y_1+1)^2+\sum_{i=2}^dy_i^2\leq 1}$, hence we have:
$$\begin{matrix}\delta (d)&\geq (c_0+\frac{1}{2})^{d-1}\left((c_0+\frac{1}{2})(c_0-y_1)+(1+y_1)^2-1\right)\\
&\geq (c_0+\frac{1}{2})^{d-1}\left((y_1+ \frac{3}{4}-\frac{c_0}{2})^2+\frac{3}{4}c_0^2+\frac{5}{4}c_0-\frac{9}{16}\right).
\end{matrix}$$
As $\frac{3}{4}c_0^2+\frac{5}{4}c_0-\frac{9}{16}=0$, we conclude that $c_0\|.\|^2-\eta$ is positive semidefinite.\\

Let us now prove that $(1+c_0){\bf 1}+S$ is    positive semidefinite. We compute
$$(1+c_0){\bf 1}+S=\begin{pmatrix}1+c_0+y_1&y_2&y_3&\dots &y_{d-1}&y_d\\
y_2& \frac{1}{2}+c_0&0&\dots &0&0\\
.&.&.&\dots&.&.\\
y_d&0&0&\dots&0& \frac{1}{2}+c_0
\end{pmatrix}
$$
Then the restriction of $\eta+(1+c_0)\|.\|^2$ to $\{ 0\}\times \R^{d-1}$  is positive definite and we just have to prove that $\det ((1+c_0){\bf 1}+S)$ is non negative.  Using the computations that we did for $\delta (d)$ (we replace $y_i$ by $-y_i$ and $y_1$ by $-(1+y_1)$), we obtain:
$$\det((1+c_0){\bf 1}+S)=(c_0+\frac{1}{2})^{d-1}\left((c_0+\frac{1}{2})(c_0+1+y_1)-\sum_{i=2}^dy_i^2\right).$$
We have proved that $\displaystyle{1\geq \sum_{i=1}^d y_i^2}$ hence we deduce
$$\begin{matrix}\det((1+c_0){\bf 1}+S)&\geq (c_0+\frac{1}{2})^{d-1}\left((c_0+\frac{1}{2})(c_0+1+y_1)+y_1^2-1\right)\hfill\\
&\geq(c_0+\frac{1}{2})^{d-1}\left( (y_1+\frac{c_0}{2}+\frac{1}{4})^2+\frac{3}{4}c_0^2+\frac{5}{4}c_0-\frac{9}{16}\right)\hfill\\
&\geq(c_0+\frac{1}{2})^{d-1}(y_1+\frac{c_0}{2}+\frac{1}{4})^2.\hfill
\end{matrix}
$$
Then the quadratic form $(1+c_0)^2+\eta$ is positive semidefinite.

\end{demo}

\end{document}